%% file: CategoryOfFilters.tex
\pdfoutput=1
\documentclass[12pt,reqno]{amsart}
\usepackage{amssymb,latexsym}
\usepackage{enumerate}
\usepackage{amscd}
\usepackage{verbatim}
\usepackage[arrow,matrix,curve]{xy}

\usepackage[english]{babel}
\usepackage{etoolbox}
\usepackage{csquotes}
\usepackage[backend=biber,style=numeric]{biblatex}
\usepackage[user,xr]{zref}
\usepackage{xr-hyper}
\usepackage[colorlinks]{hyperref}
\newcommand{\mylabel}[1]{\label{#1}\zlabel{#1}}

\zexternaldocument[I-]{GenRelGenEq}
\zexternaldocument[II-]{NonsymmetricClosedCategories}

\theoremstyle{plain}
\newtheorem{theorem}{Theorem}
\newtheorem{corollary}[theorem]{Corollary}
\newtheorem{lemma}[theorem]{Lemma}
\newtheorem{proposition}[theorem]{Proposition}

\theoremstyle{definition}
\newtheorem{definition}[theorem]{Definition}
\newtheorem{example}[theorem]{Example}

\newtheorem{remark}[theorem]{Remark}
\newtheorem*{unotation}{Notation}
\newtheorem{notation}{Notation}

\numberwithin{equation}{section}
\numberwithin{theorem}{section}
\numberwithin{assumptions}{section}
\numberwithin{notation}{section}

\def\pair#1#2{\langle #1, #2\rangle}

\def\Mof#1{\mathcal M(A)}

\def\Set{{\text{\bf Set}}}

\def\Ab{{\text{\bf Ab}}}

\def\Pre{{\text{\bf Pre}}}

\def\Top{{\text{\bf Top}}}

\def\Unif{{\text{\bf Unif}}}

\def\HUnif{{\text{\bf HUnif}}}
\def\CUnif{{\text{\bf CUnif}}}
\def\CHUnif{{\text{\bf CHUnif}}}

\def\LPartial{{\text{\bf LPartial}}}
\newcommand{\Ind}{\text{\bf Ind}}

\newcommand{\Partial}{\text{\bf Partial}}

\newcommand{\SuperUnif}{\text{\bf SupUnif}}

\newcommand{\CatFil}{\text{\bf Fil}}

\newcommand{\Fg}{\operatorname{Fg}}
\newcommand{\Up}{\operatorname{Up}}

\newcommand{\Sub}{\operatorname{Sub}}

\newcommand{\Fil}{\operatorname{Fil}}

\newcommand{\dom}{\operatorname{dom}}
\newcommand{\core}{\operatorname{core}}

\newcommand{\nullset}{\{\}}

\newcommand{\dd}{\operatorname{dd}}
\newcommand{\rr}{\operatorname{r}}

\newcommand{\opsquare}{{\mathop\square}}

\def\congruence{on\-gru\-ence\discretionary{-}{}{-}}
\def\conM/{c\congruence mod\-u\-lar}
\def\ConM/{C\congruence mod\-u\-lar}
\def\conD/{c\congruence dis\-trib\-u\-tive}
\def\ConD/{C\congruence dis\-trib\-u\-tive}
\def\conP/{c\congruence per\-mut\-a\-ble}
\def\ConP/{C\congruence per\-mut\-a\-ble}
\def\conMity/{\conM/\-i\-ty}
\def\ConMity/{\ConM/\-i\-ty}
\def\conDity/{c\congruence dis\-trib\-u\-tiv\-i\-ty}
\def\ConDity/{C\congruence dis\-trib\-u\-tiv\-i\-ty}
\def\conPity/{c\congruence per\-mut\-a\-bil\-i\-ty}
\def\ConPity/{C\congruence per\-mut\-a\-bil\-i\-ty}
\def\usprv/{un\-der\-ly\-ing-set-pre\-ser\-ving}

\def\ie/{{i.e.}}
\def\Ie/{{I.e.}}
\def\eg/{{e.g.}}
\def\Eg/{{E.g.}}
\def\etc/{{etc.}}

\newdimen\mysubdimen
\newbox\mysubbox

\def\subwhat#1#2#3{{
\setbox\mysubbox=\hbox{#3}
\mysubdimen=\wd\mysubbox
\setbox\mysubbox=\hbox{$#1#2$}
\ifnum\mysubdimen>\wd\mysubbox
\vtop{
\hbox to\mysubdimen{\hfil\box\mysubbox\hfil}
\nointerlineskip
\hbox{#3}}
\else
\mysubdimen=\wd\mysubbox
\vtop{
\box\mysubbox
\nointerlineskip
\hbox to\mysubdimen{#3}}
\fi
}}

\begin{document}

\title[III. The Closed Category of Filters]{Elements of Topological Algebra \\ \  \\ III. The Closed Category of Filters}
\author{William H. Rowan}
\address{PO Box 20161 \\
         Oakland, California 94620}
\email{william.rowan@ncis.org}

\date{\today}

\begin{abstract}
\input abstract.tex
\end{abstract}

\maketitle

\section*{Introduction}

In this paper, we study the \emph{category of filters}, defined almost exactly as defined in \cite{kr70}, the only difference being  that we admit as objects in the category of filters, $\CatFil$,  filters  that contain the empty set.  This necessitates that we define germs of functions using explicit partial functions. 

The point of the paper is not this minor change, but the definition we give of  nonsymmetric closed category structure (\cite{reltclosed} - see \cite[VII.1 and~VII.7]{macl2} for the canonical treatment of the symmetric case) on the category $\CatFil$. We feel the need for a self-contained definition and exploration of the properties of this category, to facilitate forthcoming
more
 detailed explorations of applications, such as \cite{reltunifonfil} and \cite{reltgeneq}, briefly mentioned in Section~\ref{S:Applications}.

\section{Filters}\mylabel{S:Filters}

In this section, we will begin to define $\CatFil$, the category of filters and germs, with a discussion of \emph{filters}.
Recall

\begin{definition}
A \emph{filter $\mathcal F$ on a set $S$}\index{filter!filter on a set} is a set of subsets of $S$ such that
\begin{enumerate}
\item $S\in\mathcal F$;
\item if $F\in\mathcal F$ and $F\subset F'\subseteq S$, then $F'\in\mathcal F$;
and
\item if $F$, $F'\in\mathcal F$ then $F\cap F'\in\mathcal F$.
\end{enumerate}
We will denote the set of filters on $S$ by \index{filter!$\Fil S$ for a set $S$}$\Fil S$.
Note that some definitions include another condition: that $\mathcal F$ be \emph{proper}, i.e., that $\mathcal F$ not contain the empty set.  However, when we assume this, we shall explicitly call the filter a \emph{proper filter}\index{filter!proper filter}.
\end{definition}

\begin{remark}
A filter $\mathcal F\in\Fil S$ uniquely determines $S$, as $S=\bigcup\mathcal F$.
\end{remark}

\subsection{Ordering the set of filters}
$\Fil S$ admits a partial ordering, which we (unlike some authors) take to be  \emph{reverse inclusion}:

\begin{proposition} Let $\mathcal F$, $\mathcal G\in\Fil S$. The following are equivalent (and, if they hold, we will say $\mathcal F\leq\mathcal G$):
\begin{enumerate}
\item $\mathcal G\subseteq\mathcal F$, and
\item For any $G\in\mathcal G$, there is an $F\in\mathcal F$ with $F\subseteq G$.
\end{enumerate}
\end{proposition}

\subsection{Filter bases}
\begin{definition} We say that a set $\mathcal B$ of subsets of $S$ is a \index{filter base}\emph{filter base} or \index{base for a filter}\emph{base for a filter} if $F$, $F'\in\mathcal B$ imply there is a $\bar F\in\mathcal B$ such that $\bar F\subseteq F\cap F'$.
\end{definition}

\begin{definition} If $\mathcal B$ is a filter base (of subsets of $S$), then the set of subsets $\mathcal F=\{\,F\subseteq S\mid\exists B\in\mathcal B\text{ such that }B\subseteq F\,\}$ is a filter $\mathcal F$, which is the least (in the above ordering) such that $\mathcal B\subseteq\mathcal F$, and which we denote by $\Fg_S\mathcal B$ or simply $\Fg\mathcal B$. In this case, we also say that $\mathcal B$ is a \emph{base for $\mathcal F$}. If $\mathcal B$ is not a filter base, then the least filter $\mathcal F$ such that $\mathcal B\subseteq \mathcal F$, and which we still denote by $\Fg\mathcal B$, is $\Fg\mathcal B'$, where $\mathcal B'$ is the set of finite intersections of elements of $\mathcal B$ (and is a filter base). In either case, we say that $\mathcal F=\Fg\mathcal B$ is \emph{the filter generated by $\mathcal B$.}
\end{definition}

\subsection{Subfilters} If $\mathcal F\in\Fil S$, then we say that a filter $\mathcal F'\in\Fil S$ is a \index{subfilter}\emph{subfilter of $\mathcal F$} if 
$\mathcal F'\leq\mathcal F$.
We will denote by \index{filter!$\Fil F$ for a filter $\mathcal F$}$\Fil\mathcal F$ the set of subfilters $\mathcal F'\leq\mathcal F$, i.e., the interval sublattice $I_{\Fil S}[\bot,\mathcal F]$.

\subsection{\texorpdfstring{$\Fil S$}{Fil S} is coalgebraic}
Recall that a lattice is \emph{coalgebraic} if its dual is \emph{algebraic} \cite[Definition~\zref{I-D:Algebraic}]{reltgeneq}.

\begin{proposition} Let $\mathcal F$ be a filter.  Then $\Fil\mathcal F$ is a coalgebraic complete lattice, where
\begin{enumerate}
\item $\bigvee_i\mathcal F_i=\bigcap_i\mathcal F_i$ and
\item $\bigwedge_i\mathcal F_i=\Fg\{\bigcup_i\mathcal F_i\}$.
\end{enumerate}
\end{proposition}

\subsection{\texorpdfstring{$f(\mathcal F)$ and $f^{-1}(\mathcal G)$}{PushPullFilterAlongFunction}}

\begin{definition}\mylabel{D:PushingFilters}
Let $S$ and $T$ be sets, and $f:S\to T$ a function. If $\mathcal F\in\Fil S$, then we define
\[f(\mathcal F)=\Fg\{\,f(F)\mid F\in\mathcal F\,\};\]
if $\mathcal G\in\Fil T$, then we define
\[f^{-1}(\mathcal G)=\Fg\{\,f^{-1}(G)\mid G\in\mathcal G\,\}.\]
\end{definition}

\begin{proposition} We have
\begin{enumerate}
\item If $f:S\to T$ is a function, $\mathcal F\in\Fil S$, and $\mathcal G\in\Fil T$, then 
\[f(\mathcal F)\leq\mathcal G\iff\mathcal F\leq f^{-1}(\mathcal G).\]
\item 
If in addition $g:T\to W$ is a function and $\mathcal H\in\Fil W$, then
\[g(f(\mathcal F))=(gf)(\mathcal F)\]
and
\[f^{-1}(g^{-1}(\mathcal H))=(gf)^{-1}(\mathcal H).\]
\end{enumerate}
\end{proposition}

\section{Partial
Functions; Restriction}\mylabel{S:Partial}

If $S$, $T$ are sets, a \emph{partial function} $f:S\to T$ is a
method or rule $f$ which somehow assigns an element $f(s)\in T$ to
$s$, for some, but not necessarily all, elements $s\in S$.

\begin{definition} We
denote the \emph{domain of definition} of $f$, the subset
of $s\in S$ such that $f(s)$ is defined, by $\dd(f)$. We
denote the \emph{range} of $f$, the set of elements of
the form $f(s)$ for some $s\in S$, by $\rr(f)$.
\end{definition}

\begin{definition}
If $f:S\to T$, $g:T\to W$ are partial functions, then the \emph{composite partial function} of $f$ and $g$, denoted $g\circ f$, is the partial function that assigns an element $s\in S$ to $g(f(s))$, if both $s\in\dd(f)$ and $f(s)\in\dd(g)$.
\end{definition}

\begin{definition}
If $f$ and
$g$ are partial functions from $S$ to $T$, then we say that
$f$ is a
\emph{restriction} of $g$ if
$\dd(f)\subseteq\dd(g)$ and
$f(s)=g(s)$ for $s\in\dd(f)$.
If $f$ is a partial function on $S$, and $D\subseteq S$, then we denote by $f|_D$ the \emph{restriction of $f$ to $D$}, i.e., the rule which assigns $f(s)$ to $s$ for $s\in D\cap\dd(f)$ and does not assign anything to elements not in $D\cap\dd(f)$.
\end{definition}

\subsection{\texorpdfstring{$f(D)$ and $f^{-1}(D)$ when $f$ is a Partial Function}{PushPullSubsetAlongPartial}}
If $f:S\to T$ is a partial function, and $D\subseteq S$, we define
\[f(D)=\{\,f(s)\mid s\in D\cap\dd(f)\,\}\]
and if $D'\subseteq T$,
\[f^{-1}(D')=\{\,s\in S\mid s\in\dd(f)\text{ and }f(s)\in D'\,\}.\]

\begin{lemma}\mylabel{T:DDLemma} Let $f:S\to T$, $g:T\to W$  be  partial functions. Then $\dd(g\circ f)=f^{-1}(\dd(g))$.
\end{lemma}

\begin{proof} Referring to the definitions, we have
\begin{align*}s\in\dd(g\circ f)
&\iff s\in\dd(f)\text{ and } f(s)\in\dd(g)\\
&\iff s\in f^{-1}(\dd(g)).
\end{align*}
\end{proof}

\begin{lemma}
\mylabel{T:RestrX}
Let $f$ and $g$ be partial functions from $S$ to $T$.
If $f=g|_{\hat D}$ for some $\hat D\subseteq S$ then
\begin{enumerate}
\item
If $D\subseteq S$, then $f(D)\subseteq g(D)$, with $f(D)=g(D)$ when $D\subseteq\hat D$, and
\item if $D'\subseteq T$, then $f^{-1}(D')\subseteq g^{-1}(D')$.
\end{enumerate}
\end{lemma}

\begin{proof}
(1): We have
\begin{align*}t\in f(D)
&\iff t\in g|_{\hat D}(D)\\
&\iff \exists s\in\hat D\cap D\cap\dd(g)\text{ such that }t=g(s)\\
&\implies  \exists s\in D\cap\dd(g)\text{ such that }t=g(s)\\
&\iff s\in g(D), 
\end{align*}
with equivalence if $D\subseteq\hat D$.

(2): We have
\begin{align*}
s\in f^{-1}(D')&\iff s\in\dd(f)\text{ and }f(s)\in D'\\
&\iff s\in\dd(g|_{\hat D})\text{ and }g|_{D'}(s)\in D'\\
&\iff s\in\hat D\cap\dd(g)\text{ and }g(s)\in D'\\
&\implies s\in\dd(g)\text{ and }g(s)\in D'\\
&\iff s\in g^{-1}(D').
\end{align*}
\end{proof}

\begin{proposition}\mylabel{T:GaloisParFnSubsets} We have
\begin{enumerate}
\item
If $f:S\to T$ is a partial function, $D\subseteq S$, and $D'\subseteq T$, then
\[f(D)\subseteq D'\iff D\subseteq (S-\dd(f))\cup f^{-1}(D');\]
\item  if in addition, there is another partial function $g:T\to W$, and $D''\subseteq W$, then
\[g(f(D))=(g\circ f)(D)\]
and
\[f^{-1}(g^{-1}(D''))=(g\circ f)^{-1}(D'').\]
\end{enumerate}
\end{proposition}

\begin{proof} 

 (1) We have
 \begin{align*}
f(D)\subseteq D'
&\iff (s\in D\text{ and }s\in\dd(f)\implies f(s)\in D')\\
&\iff (s\notin D\text{ or }s\notin\dd(f)\text{ or }f(s)\in D')\\
&\iff (s\in D\implies(s\notin\dd(f)\text{ or }f(s)\in D'))\\
&\iff D\subseteq (S-\dd(f))\cup f^{-1}(D');\text{ and}
\end{align*}

(2): we have
\begin{align*}
w\in g(f(D))
&\iff \exists s\text{ such that }s\in\dd(f)\text{ and }f(s)\in\dd(g)\text{ and }w=g(f(s))\\
&\iff \exists s\text{ such that }s\in\dd(g\circ f)\text{ and }(g\circ f)(s)=w\\
&\iff w\in(g\circ f)(D);\text{ and }
\end{align*}

\begin{align*}
s\in f^{-1}(g^{-1}(D''))
&\iff s\in\dd(f)\text{ and }f(s)\in g^{-1}(D'')\\
&\iff s\in\dd(f)\text{ and }f(s)\in\dd(g)\text{ and }g(f(s))\in D''\\
&\iff s\in\dd(g\circ f)\text{ and }(g\circ f)(s)\in D''\\
&\iff s\in(g\circ f)^{-1}(D'').
\end{align*}
\end{proof}

\subsection{\texorpdfstring{$f(\mathcal F)$ and $f^{-1}(\mathcal G)$ when $f$ is a Partial Function}{PushPullFilterAlongPartial}}
If
$S$ and
$T$ are sets,
$f:S\to T$ is a partial function, and $\mathcal F$ is a filter of subsets of $S$,
then we define
\[f(\mathcal F)=\Fg\{\,f(F)\mid F\in\mathcal F\,\}.\]
On the other hand, given $f$ and a filter $\mathcal G$ of subsets of $T$, then we define
\[f^{-1}(\mathcal G)=\Fg\{\,f^{-1}(G)\mid G\in\mathcal G\,\}.\]

\begin{remark}\mylabel{R:FBtoFB}
Note that for a total function $f$ (i.e. if $\dd(f)=S$), this definition coincides with the definition (Definition~\ref{D:PushingFilters}) given previously.
Also, note that the mappings $D\mapsto f(D)$ and $D'\mapsto g^{-1}(D')$ are monotone, and, consequently, take filter bases to filter bases.
\end{remark}

\begin{theorem}\mylabel{T:GaloisOne} We have
\begin{enumerate}
\item If $f$ is a partial function from $S$ to $T$,
$\mathcal F$ is a filter of subsets of $S$ such that $\dd(f)\in\mathcal F$, and
$\mathcal G$ is a filter of subsets of $T$, then
\[f(\mathcal F)\leq \mathcal G \mathrel{\text{iff}}
\mathcal F\leq f^{-1}(\mathcal G);\]
\item and if in addition, we have a partial function $g:T\to W$, and $\mathcal H$ is a filter of subsets of $W$, then
\[g(f(\mathcal F))=(g\circ f)(\mathcal F)\]
and
\[f^{-1}(g^{-1}(\mathcal H))=(g\circ f)^{-1}(\mathcal H).\]
\end{enumerate}
\end{theorem}

\begin{proof} 
First, we note that by Remark~\ref{R:FBtoFB}, if we have $f:S\to T$ and $\mathcal F\in\Fil S$, then 
\begin{align*}
f(\mathcal F)&=\Fg\{\,f(F)\mid F\in\mathcal F\,\}\\
&=\Up\{\,f(F)\mid F\in\mathcal F\,\}
\end{align*}
where $\Up B$, for a subset $B$ of a lattice (in this case, the lattice of subsets of $S$), denotes the set of elements of the lattice greater than or equal to an element in $B$.
Similarly, if $\mathcal G\in\Fil T$, then
\begin{align*}
f^{-1}(\mathcal G)&=\Fg\{\,f^{-1}(G)\mid G\in\mathcal G\,\}\\
&=\Up\{\,f^{-1}(G)\mid G\in\mathcal G\,\}.
\end{align*}
Then, to prove the statements of the Theorem, we consider that

 (1):
by Proposition~\ref{T:GaloisParFnSubsets}, and since $\dd(f)\in\mathcal F$,
\begin{align*}
f(\mathcal F)\leq\mathcal G&\iff\Fg\{\,f(F)\mid F\in\mathcal F\,\}\leq\mathcal G\\
&\iff\forall G\in\mathcal G,\exists F\in\mathcal F\text{ such that }f(F)\subseteq G\\
&\iff\forall G\in\mathcal G,\exists F\in\mathcal F\text{ such that }F\subseteq (S-\dd(f))\cup f^{-1}(G)\\
&\iff\forall G\in\mathcal G,\exists F\in\mathcal F\text{ such that }F\cap\dd(f)\subseteq  f^{-1}(G)\\
&\iff\forall G\in\mathcal G,\exists F\in\mathcal F\text{ such that }F\subseteq  f^{-1}(G)\\
&\iff\mathcal F\leq \Fg\{\,f^{-1}(G)\mid G\in\mathcal G\,\}\\
&\iff\mathcal F\leq f^{-1}(\mathcal G);
\end{align*}

(2): we also have
\begin{align*}
H\in g(f(\mathcal F))
&\iff H\in\Fg\{\,g(G)\mid G\in\Fg\{\,f(F)\mid F\in\mathcal F\,\}\,\}\\
&\iff H\in\Fg\{\,g(f(F))\mid F\in\mathcal F\,\}\\
&\iff H\in(g\circ f)(\mathcal F), \text{ and}
\end{align*}
\begin{align*}F\in f^{-1}(g^{-1}(\mathcal H))
&\iff F\in \Fg\{\,f^{-1}(G)\mid G\in\Fg\{\,g^{-1}(H)\mid H\in\mathcal H\,\}\,\}\\
&\iff F\in\Fg\{\,f^{-1}(g^{-1}(H))\mid H\in\mathcal H\,\}\\
&\iff F\in(g\circ f)^{-1}(\mathcal H).
\end{align*}
\end{proof}

\section{\texorpdfstring{The Category $\LPartial$}{LPartial}}\mylabel{S:LPartial}

\subsection{Admissible domains of definition}
 We
would like to consider as equivalent, functions (or partial functions) on a set $S$
which have a common restriction, and to work with the
resulting equivalence classes of partial functions that
we will call
\emph{germs} in Section~\ref{S:Germs}. In this plain form, the equivalence relation is uninteresting, because any two partial
functions  on $S$ have a common restriction to the empty set.
For this reason, we will limit this relation of having a common restriction to considering two partial functions on $S$ as equivalent, if and only if they have a common restriction to a subset of $S$  that belongs to a specified set of admissible
domains of definition. In order that this result in an equivalence relation, we will require the specified set of admissible domains of definition be a filter of subsets of $S$. We call partial
functions from $S$ to $T$, defined on some set in the filter $\mathcal F\in\Fil S$,
\emph{admissible partial functions from $\mathcal F$ to $T$}. We already saw, in Theorem~\ref{T:GaloisOne}(1), a use of the condition that a partial function be admissible, although we didn't yet call it that. 

\begin{notation} \mylabel{N:Partial}
If $\mathcal F$ is a filter on some set $S$, $\mathcal G$ is a filter on some set $T$, and $G\in\mathcal G$, we denote by $\Partial(\mathcal F,\mathcal G)$ ($\LPartial(\mathcal F,\mathcal G)$), the set of admissible partial functions from $\mathcal F$ to $\mathcal G$, (respectively the set of \emph{local} admissible partial functions from $\mathcal F$ to $\mathcal G$) and by $\Partial(\mathcal F,\mathcal G,G)$ ($\LPartial(\mathcal F,\mathcal G,G)$) the set of admissible partial functions from $S$ to $T$ (respectively the set of \emph{local} admissible partial functions from $\mathcal F$ to $\mathcal G$), such that for some $F\in\mathcal F$, $f(F)\subseteq G$.
\end{notation}

\begin{remark} These notations will be useful not only for defining (in this section and the next) the categories $\LPartial$ and $\CatFil$, but also for defining a nonsymmetric closed structure on $\CatFil$ in Sections~\ref{S:FilMonoidal} and~\ref{S:FilClosed}.
Note that elements of $\LPartial(\mathcal F,\mathcal G,G)$ are not necessarily arrows in any category, although they play a role in the definitions of categories in the later sections we mentioned.
\end{remark}

\subsection{Locality} In order for admissible partial functions, and germs of admissible partial
functions, to be the arrows of categories, we
will need to impose another condition,
\emph{locality}, that the partial functions must
satisfy. 
The problem is that if $f$ and $g$ are partial functions,
$\dd(g\circ f)$ may be so small that $g\circ f$ is not admissible. Thus, suppose $f$ is an admissible partial function from $\mathcal F$ to $T$, where $\mathcal F\in\Fil S$, and $g$ is an admissible partial function from $\mathcal G$ to $W$, where $\mathcal G\in\Fil T$.
As
$\dd(g\circ f)=f^{-1}(\dd(g))$ and $\dd(g)$ could be any element
of
$\mathcal G$, what we want to require of $f$ is that for
any $G\in\mathcal G$,
$f^{-1}(G)\in\mathcal F$.  For, if that is true, then because $\mathcal F$ is a filter, $g\circ f$ will defined on $\dd(f)\cap f^{-1}(G)\in\mathcal F$ and will be admissible.

\begin{definition}
We say that
$f$ is
\emph{local} (with respect to $\mathcal G$) if for each
$G\in\mathcal G$, $f^{-1}(G)\in\mathcal F$, or in
other words, there is an
$F\in\mathcal F$ such that
$F\subseteq\dd(f)$ and $f(F)\subseteq G$.
\end{definition}

\begin{proposition}\mylabel{T:AdmissibleLocal} 
If $f:S\to T$ is a partial function, admissible with respect to $\mathcal F$ and local with respect to $\mathcal G$, and $g:T\to W$ is a partial function, admissible with respect to $\mathcal G$ and local with respect to $\mathcal H$, then $g\circ f:S\to W$ is a partial function, admissible with respect to $\mathcal F$ and local with respect to $\mathcal H$.
\end{proposition}

\begin{definition} We denote the set of partial functions from $S=\bigcup\mathcal F$ to $T=\bigcup\mathcal G$, admissible with respect to $\mathcal F$ and local with respect to $\mathcal G$, by $\LPartial(\mathcal F,\mathcal G)$.  This defines a category $\LPartial$, where the identity arrow from $\mathcal F$ to itself is just the identity function on $S=\bigcup\mathcal F$.
\end{definition}

\begin{proof} [Proof that $\LPartial$ is a category] If $f:\mathcal F\to\mathcal G$ and $g:\mathcal G\to\mathcal H$, then $g\circ f$ is defined as the partial function with domain of definition $\dd(f)\cap f^{-1}(g)$, sending $s\in\dd(f)$ to $g(f(s))$.  That is, it is the partial function corresponding to the relational product of $f$ and $g$, seen as relations.
The axioms of a category are immediate.
\end{proof}

\section{\texorpdfstring{Germs and the Category $\CatFil$}{Germs}}\mylabel{S:Germs}

 \subsection{Germs of partial functions}
 
\begin{definition}\mylabel{D:GAPF} If $\mathcal F$ is a filter of subsets of a set $S$, then a
\emph{germ of admissible partial functions from $\mathcal F$ to a set $T$} is an $\equiv_{\mathcal F}$-equivalence class of such partial functions, where $f\equiv_{\mathcal F}g$ iff for some $F\in\mathcal F$, $\dd(f)\cap F=\dd(g)\cap F$ and $f(s)=g(s)$ for all $s\in\dd(f)\cap\dd(g)\cap F$.

If $f$ is a partial function, we will denote its germ (the $\equiv_{\mathcal F}$-equivalence class containing $f$) by $\Gamma f$, or by $f/\mathcal F$.  We will also use $\Gamma$ as a set-function, so that if $Y$ is a set of admissible partial functions from $\mathcal F$ to $T$, $\Gamma(Y)$ will denote the set of germs of the partial functions in $Y$.  We will use $\Gamma$ in this way particularly in two cases:  we will shortly define the hom-set $\CatFil(\mathcal G,\mathcal H)=\Gamma(\LPartial(\mathcal H,\mathcal G))$ of the category $\CatFil$, and we will later define an internal hom-object $\mathcal G^{\mathcal H}=\Fg\{\,\Gamma(\Partial(\mathcal H,\mathcal G,G)\,\}$ for the category $\CatFil$ using the base of sets $\Gamma(\Partial(\mathcal H,\mathcal G,G)$.
\end{definition}

\begin{theorem}\mylabel{T:AnyRep} Let  $f$, $g$ be partial functions from $S$ to $T$, and  $\mathcal F$ a filter of subsets of $S$ such that $f\equiv_{\mathcal F}g$. We have
\begin{enumerate}
\item If $\mathcal F'$ is a subfilter of $\mathcal F$, then $f(\mathcal F')=g(\mathcal F')$, and
\item if $\mathcal G$ is a filter of subsets of $T$, then $f^{-1}(\mathcal G)\wedge\mathcal F= g^{-1}(\mathcal G)\wedge\mathcal F$
\end{enumerate}
\end{theorem}

\begin{proof} 
 Let $F\in\mathcal F$ be such that $\dd(f)\cap F=\dd(g)\cap F$ and $f=g$ on $\dd(f)\cap\dd(g)\cap F$. Then

(1): Since $F\in\mathcal F$, and $\mathcal F'\leq\mathcal F$, there is an $F'\in\mathcal F'$ such that $F'\subseteq F$. We have $\dd(f)\cap F'=\dd(g)\cap F'$ and $f=g$ on $\dd(f)\cap\dd(g)\cap F'$. Thus,
\begin{align*}
f(\mathcal F')&=\Fg\{\,f(F')\mid F'\in\mathcal F'\,\}\\
&=\Fg\{\,f(F')\mid F'\in\mathcal F'\text{ and }F'\subseteq F\,\}\\
&=\Fg\{\,g(F')\mid F'\in\mathcal F'\text{ and }F'\subseteq F\,\}\\
&=\Fg\{\,g(F')\mid F'\in\mathcal F'\,\}\\
&=g(\mathcal F').
\end{align*}

(2): If $F\in\mathcal F$ is such that $\dd(f)\cap F=\dd(g)\cap F$ and $f(s)=g(s)$ for all $s\in\dd(f)\cap\dd(g)\cap F$, then the same statement is true for any smaller $F$. Consequently, we have

\begin{align*} f^{-1}(\mathcal G)\wedge\mathcal F&=\Fg\{\,f^{-1}(G)\cap F''\mid G\in\mathcal G, F''\in\mathcal F\,\}\\
	&=\Fg\{\,g^{-1}(G)\cap F''\mid G\in\mathcal G, F''\in\mathcal F\,\}\\
	&=g^{-1}(\mathcal G)\wedge\mathcal F;
\end{align*}
\end{proof}

\begin{notation}
 We continue to use
roman letters
$f$, $g$, etc.\ to denote partial functions, and will
use greek letters
$\varphi$,
$\gamma$, etc.\  for germs.
\end{notation}

\subsection{\texorpdfstring{$\varphi(\mathcal F)$ and $\varphi^{-1}(\mathcal G)$}{PushPullFilterAlongGerm}}

\begin{notation}\label{N:PushPullDef}
Let $\mathcal F$, $\mathcal F'$ be filters of subsets of $S$ such that $\mathcal F'\leq\mathcal F$, and let $\mathcal G$ be a filter of subsets of $T$.  If $\varphi$ is a germ of partial functions from $S$ to $T$ and admissible wrt $\mathcal F$, then we define
\[\varphi(\mathcal F')= f(\mathcal F')\]
and
\[\varphi^{-1}(\mathcal G)=f^{-1}(\mathcal G)\wedge\mathcal F,\]
where $f$ is any admissible partial function representing $\varphi$. By Theorem~\ref{T:AnyRep}, these formulae are independent of the choice of $f$.
\end{notation}

\begin{proposition} Let $\mathcal F\in\Fil S$, and $\mathcal G\in\Fil T$.  If $\varphi$ is a germ of partial functions admissible wrt $\mathcal F$, then the following are equivalent:
\begin{enumerate}
\item For some admissible partial function $f:\mathcal G\to T$ representing $\varphi$, $f$ is local with respect to $\mathcal G$;
\item For every admissible partial function $f:\mathcal G\to T$ representing $\varphi$, $f$ is local with respect to $\mathcal G$;
\item $\varphi(\mathcal F)\leq\mathcal G$.
\end{enumerate}
\end{proposition}

\begin{proof} Certainly $(2)\implies(1)$.

$(1)\implies(3)$: Let $f:\mathcal G\to T$ be an admissible partial function representing $\varphi$, local with respect to $\mathcal G$, and suppose that we are given $G\in\mathcal G$. Since $f$ is local wrt $\mathcal G$, there is an $F\in\mathcal F$ such that $f(F)\subseteq G$.  This shows that $\varphi(\mathcal F)=f(\mathcal F)\leq\mathcal G$.

$(3)\implies(2)$: Assume that $\varphi(\mathcal F)\leq\mathcal G$, and let $f:\mathcal G\to T$ be an admissible partial function representing $\varphi$.  Let $G$ be any element of $\mathcal G$. Since $\varphi(\mathcal F)\leq\mathcal G$, there is an $f\in\mathcal F$ such that $f(F)\subseteq G$. Since $G$ was any element of $\mathcal G$, this shows that $f$ is local wrt $\mathcal G$.

\end{proof}

\begin{theorem} Let $\mathcal F$ be a filter of subsets of a set $S$, and let $\mathcal G$ be a filter of subsets of another set $T$. We have
\begin{enumerate}
\item $\equiv_{\mathcal F}$ is an equivalence relation on the set of partial functions from $S$ to $T$;
\item if $f\equiv_{\mathcal F}g$, then $f$ is admissible wrt $\mathcal F$ iff $g$ is admissible wrt $\mathcal F$;
\item if $f\equiv_{\mathcal F}g$, then $f$ is local wrt $\mathcal G$ iff $g$ is local wrt $\mathcal G$;
\item $\equiv_{\mathcal F}$ is an equivalence relation on the set of partial functions from $S$ to $T$ local wrt $\mathcal G$;
\item If $f$ and $g$ are admissible (wrt $\mathcal F$) functions from $S$ to $T$, and $f\equiv_{\mathcal F}g$, then there is an $F\in\mathcal F$ such that $F\subseteq\dd(f)\cap\dd(g)$ and $f|_F=g|_F$.
\item If $f$, $f':S\to T$ are partial functions, admissible with respect to $\mathcal F$ and local with respect to $\mathcal G$, with $f\mathrel{\equiv_{\mathcal F}}f'$, and $g$, $g'\in\CatFil(\mathcal G,\bigcup\mathcal H)$ are admissible with respect to $\mathcal H$, with $g\mathrel{\equiv_{\mathcal G}}g'$, then $(g\circ f)\mathrel{\equiv_{\mathcal F}}(g'\circ f')$.
\end{enumerate}
\end{theorem}

\begin{proof}
(1): Let $f\equiv_{\mathcal F}g\equiv_{\mathcal F}h$. Then there is an $F\in\mathcal F$ such that $\dd(f)\cap F=\dd(g)\cap F$ and $f=g$ on $\dd(f)\cap\dd(g)\cap F$, and an $F'\in\mathcal F$ such that $\dd(g)\cap F'=\dd(h)\cap F'$ and $g=h$ on $\dd(g)\cap\dd(h)\cap F'$. Then $F\cap F'\in\mathcal F$, $\dd(f)\cap F\cap F'=\dd(g)\cap F\cap F'$, $\dd(g)\cap F\cap F'=\dd(h)\cap F'$, and $f=h$ on $\dd(f)\cap\dd(g)\cap\dd(h)\cap F\cap F'=\dd(f)\cap\dd(h)\cap F\cap F'$. Thus, $\equiv_{\mathcal F}$ is transitive. Reflexivity and symmetricity are obvious.

(2): $f$ is admissible wrt $\mathcal F$ iff $\dd(f)\in\mathcal F$, and likewise, $g$ is admissible iff $\dd(f)\in\mathcal F$. If $f\equiv_{\mathcal F}g$, then there is an $\bar F\in\mathcal F$ such that $\dd(f)\cap \dd(g)\cap \bar F$; if $f$ is admissible so that $\dd(f)\in\mathcal F$, then $\dd(g)\cap \bar F\in\mathcal F$, which implies that $\dd(g)\in\mathcal F$. Thus, $f$ admissible implies $g$ admissible. The converse follows by symmetry.

(3): If $f\equiv_{\mathcal F}g$, then there is an $\bar F\in\mathcal F$ such that $f=g$ on $\dd(f)\cap\dd(g)\cap \bar F$, and if $f$ is local wrt $\mathcal G$, then for any $G\in\mathcal G$ there is an $f\in\mathcal F$ such that $f(F)\subseteq G$. Then $f(F\cap\bar F)\subseteq G$, which implies that $g(F\cap\bar F)\subseteq G$. Thus, $g$ is local wrt $\mathcal G$.  The converse follows by symmetry.

(4): Follows from (1) and (3).

(5): We have $\dd(f)\cap\dd(g)\in\mathcal F$, and there is an $\bar F\in\mathcal F$ such that $f$ and $g$ are equal on $\dd(f)\cap\bar F$ and $\dd(g)\cap\bar F$. We let $F=\dd(f)\cap\dd(g)\cap\bar F$.

(6): $\dd(g\circ f)=\dd(f)\cap f^{-1}(\dd(g))$. $\dd(g'\circ f')=\dd(f')\cap (f')^{-1}(\dd(g'))$. Let $F\in\mathcal F$ be such that $\dd(f)\cap F=\dd(f')\cap F$ and $f=f'$ on $\dd(f)\cap\dd(f')\cap F$, and let $G\in\mathcal G$ be such that $\dd(g)\cap G=\dd(g')\cap G$ and $g=g'$ on $\dd(g)\cap\dd(g')\cap G$. We have [perhaps we need to show $f^{-1}(\dd(g')
\cap G)\cap F=(f')^{-1}(\dd(g')
\cap G)\cap F$?]

\begin{align*}
\dd(g\circ f)\cap F\cap f^{-1}(G)
&=\dd(f)\cap f^{-1}(\dd(g))\cap F\cap f^{-1}(G)\\
&=\left[\dd(f)\cap F\right]\cap\left[ f^{-1}(\dd(g)\cap G)\cap F\right]\\
&=\left[\dd(f')\cap F\right]\cap \left[(f')^{-1}(\dd(g')\cap G)\cap F\right]\\
&=\dd(f')\cap(f')^{-1}(\dd(g'))\cap F\cap (f')^{-1}(G)\\
&=\dd(g'\circ f')\cap F\cap (f')^{-1}(G)
\end{align*}
and if $s\in F\cap f^{-1}(G)$, then
$
g(f(s))=g'(f(s))
=g'(f'(s))
$.
\end{proof}

\begin{remark} Looking at the statements of the Theorem, it makes sense to call the germ $f/\mathcal F$ of a partial function $f:S\to T$ \emph{admissible} (wrt $\mathcal F\in\Fil S$) if $f$ is admissible wrt $\mathcal F$, and \emph{local} (wrt $\mathcal G\in\Fil T$) if $f$ is local wrt $\mathcal G$.
\end{remark}

\subsection{Galois Connection}
Now we want to show that like the mappings $\mathcal F'\mapsto f(\mathcal F')$ and $\mathcal G\mapsto f^{-1}(\mathcal G)\wedge\mathcal F$, the mappings $\mathcal F'\mapsto\varphi(\mathcal F')$ and $\mathcal G\mapsto\varphi^{-1}(\mathcal G)$ constitute a Galois connection:

\begin{theorem} Let $\mathcal F$ and $\mathcal G$ be filters, on sets $S$ and $T$, respectively. Let $\mathcal F'$ be a filter such that $\mathcal F'\leq\mathcal F$. If $\varphi$ is a germ of partial functions from $S$ to $T$ admissible with respect to $\mathcal F$, then 
\[\varphi(\mathcal F')\leq\mathcal G\iff
\mathcal F'\leq \varphi^{-1}(\mathcal G).\]
\end{theorem}

\begin{proof}
Let $f$ represent $\varphi$. Then by Theorem~\ref{T:AnyRep} and the notation that follows it,
\begin{align*}
\varphi(\mathcal F')\leq\mathcal G&\iff f(\mathcal F')\leq\mathcal G\\
&\iff\mathcal F'\leq f^{-1}(\mathcal G)\wedge\mathcal F\\
&\iff\mathcal F'\leq f^{-1}(\mathcal G)\\
&\iff\mathcal F'\leq\varphi^{-1}(\mathcal G).
\end{align*}
\end{proof}

\begin{unotation} Just as we denote by $f|_F$ the restriction to $F$ of an admissible partial function $f$
we can restrict a germ $\varphi$ to a smaller subdomain filter. Thus if $\varphi=f/\mathcal F$, and $\bar{\mathcal F}\leq\mathcal F$, we can form $\varphi/\bar{\mathcal F}=(f/\mathcal F)/\bar{\mathcal F}$. because since $\bar F\leq\mathcal F$, there is an $F'\in\bar{\mathcal F}$ such that $F'\subseteq\dd f$.
\end{unotation}

\begin{proposition} In this situation, $(f/\mathcal F)/\bar{\mathcal F}=f/\bar{\mathcal F}$.\end{proposition}

\subsection{\texorpdfstring{The category $\CatFil$}{CatFil}}\label{S:CatFilSub}
If $S$, $T$ are sets, $\mathcal F\in\Fil S$, and $\mathcal G\in\Fil T$, then we denote the set of germs of partial functions from $S$ to $T$, admissible with respect to $\mathcal F$ and local with respect to $\mathcal G$, by $\CatFil(\mathcal F,\mathcal G)$. This defines $\CatFil$, the \emph{category of filters}, where the identity arrow from $\mathcal F$ to $\mathcal F$ is the germ $1_S/\mathcal F$. We denote by $\Gamma$ the functor from $\LPartial$ to $\CatFil$ that takes an admissible, local partial function $f\in\LPartial(\mathcal F,\mathcal G)$ to its germ $f/\mathcal F\in\CatFil(\mathcal F,\mathcal G)$.

\begin{remark}\mylabel{R:TotalFunction} Although we have needed to handle the boundary case that we mentioned previously, of filters $\mathcal F$ such that $\nullset\in\mathcal F$, sometimes we know that $\nullset\notin\mathcal F$. In this case, proofs can sometimes be simplified by avoiding the need to work with partial functions, for, if $\gamma\in\CatFil(\mathcal F,\mathcal G)$, and $\mathcal G$ is such that $\nullset\notin\mathcal G$, then $\gamma$ has a representative $g$ which is total, i.e. such that $\dd(g)=\bigcup\mathcal G$.
\end{remark}

\section{\texorpdfstring{Factorization of Arrows in
$\CatFil$}{S:Factorization}}\mylabel{S:Factorization}

In this section, we will define a factorization system $\pair{\mathbf E^{\CatFil}}{\mathbf M^{\CatFil}}$ for the category $\CatFil$.  See \cite[Section~\zref{I-S:FactSys}]{reltgeneq} for a discussion of this concept.
We will show that this factorization is a so-called \emph{epi, monic} factorization system. Finally, we look at the
the
$\mathbf M^\CatFil$-subobject lattice of an object $\mathcal F\in\CatFil$, and show that it can be identified
with the lattice of filters $\mathcal F'$ such that $\mathcal F'\leq\mathcal F$.

\subsection{\texorpdfstring{The factorization system $\pair{\mathbf
E^\CatFil}{\mathbf M^\CatFil}$}{FactSysCatFil}}

We define the subcategory ${\mathbf E}^\CatFil$ of $\CatFil$ to
contain all germs $\varphi:\mathcal F\to\mathcal G$ such that
$\varphi(\mathcal F)=\mathcal G$. We define the subcategory $\mathbf M^{\CatFil}$
 to contain all germs $\varphi:\mathcal F\to\mathcal G$ having the
form
$f/\mathcal F$, where $f$ is an admissible partial function one-one on its domain of definition.

\begin{theorem} \mylabel{T:CatFilFactSys}
$\pair{\mathbf
E^\CatFil}{\mathbf M^\CatFil}$ is a factorization system in
$\CatFil$, such that $\mathbf M^\CatFil$ consists of monic, and $\mathbf E^\CatFil$ of epi, arrows.
\end{theorem}

\begin{proof}

Let $\varphi:\mathcal F\to\mathcal G$, where $\mathcal F\in\Fil S$ and $\mathcal G\in\Fil T$. If $\varphi=f/\mathcal F$, then
$f$ is a local partial
function from $\mathcal F$ to $f(\mathcal F)$, and if we define $\epsilon:\mathcal F\to f(\mathcal F)$ by
$\epsilon=f/\mathcal F$ and $\mu:f(\mathcal F)\to\mathcal G$ by $\mu=1_T/f(\mathcal F)$, we have a
suitable factorization $\varphi=\mu\circ\epsilon$. For, $\epsilon\in\mathbf E^\CatFil(\mathcal F,f(\mathcal F))$ and $\mu\in\mathbf M^\CatFil(f(\mathcal F),\mathcal G)$.

If $\varphi\in\mathbf E^\CatFil(\mathcal F,\mathcal G)$, then let $\alpha$,
$\beta:\mathcal G\to\mathcal H$, and suppose that
$\alpha\circ\varphi=\beta\circ\varphi$. Let $a$, $b$, and $f$ be partial
functions representing $\alpha$, $\beta$, and
$\varphi$, where $a$ and $b$ can be taken to have $\dd(a)=\dd(b)=G\in\mathcal G$.
Let
$F\in\mathcal F$ be smaller than $\dd(f)$, such that
$f(F)\subseteq G$, and such that
$(a\circ f)|_F=(b\circ f)|_F$. Since
$(a\circ f)|_F=(b\circ f)|_F$, $a|_{f(F)}=b|_{f(F)}$, showing that
$\alpha=\beta$ because (remembering that $\varphi(\mathcal F)=\mathcal G$) $f(F)\in\mathcal G$. Thus, $\varphi$
is epi.

If $\varphi\in\mathbf M^\CatFil(\mathcal F,\mathcal G)$, let $f$ be an admissible partial
function representing $\varphi$, and such that
$f$ is one-one on $\dd(f)=F\in\mathcal F$.
Let $\alpha$, $\beta:\mathcal H\to\mathcal F$ be such that $\varphi\circ \alpha=\varphi\circ\beta$.
Let $a$ and $b$ be representatives of $\alpha$ and
$\beta$ having the same domain of definition $H\in \mathcal H$,
which is such that $a(H)\subseteq F$ and
$b(H)\subseteq F$, and such that $f\circ a=f\circ b$. (Such
representatives can always be constructed by restriction,
since $\alpha$ and $\beta$ are local and
$\varphi\circ\alpha=\varphi\circ\beta$.)
However, $f$ is one-one, implying
$a=b$, which implies that
$\alpha=\beta$. We have proved that
$\varphi$ is monic.

Suppose now that germs $\epsilon$, $\alpha$, $\beta$, and $\mu$
are given, such $\epsilon\in\mathbf E^\CatFil$ and $\mu\in\mathbf M^\CatFil$, and forming a diagram
\[\xymatrix{
&\mathcal F\ar[dl]_\epsilon \ar[dr]^\alpha \\
\mathcal W\ar[dr]_\beta &&\mathcal G\ar[dl]^\mu \\
& \mathcal
H
}\]
which commutes.
Let $\epsilon=e/\mathcal F$, $\alpha=a/\mathcal F$, $\beta=b/\mathcal W$, and
$\mu=m/\mathcal G$, where $\dd(e)=\dd(a)$, $r(e)\subseteq\dd(b)$,
$r(a)\subseteq\dd(m)$, $m$ is one-one, and furthermore
$b\circ e=m\circ a$.

Let $e$ and $a$, considered as functions from $\dd(e)$ to
$\dd(b)$ and from $\dd(a)$ to $\dd(m)$ respectively, be
factored as $e=\mathbf m[e]\circ \mathbf{\tilde e}[e]$ and $a=\mathbf m[a]\circ \mathbf{\tilde e}[a]$ in the category $\Set$. (The notations $\mathbf m[f]$, and $\mathbf{\tilde e}[f]$ for an arrow $f$, which we employ only in this proof, are defined in \cite[Section~\zref{I-S:FactSys}]{reltgeneq}; $e=\mathbf m[e]\circ\mathbf{\tilde e}[e]$ is the canonical chosen factorization of the function $e$ in the usual factorization system of the category $\Set$.) We have a commutative diagram
\[\xymatrix{
&\dd(e)=\dd(a)\ar[dl]_{\mathbf{\tilde e}[e]} \ar[dr]^{\mathbf{\tilde e}[a]} \\
\rr(e)\ar[dr]_{b\circ \mathbf m[e]}\ar@{.>}[rr]_d && \rr(a)\ar[dl]^{m\circ \mathbf m[a]}\\
&\rr(b\circ \mathbf m[e])=\rr(m\circ \mathbf m[a])
}\]
in the category $\Set$ which is uniquely diagonalized as shown by a function we denote by $d$. 
For, $\mathbf{\tilde e}[e]$ is an onto function, and $m\circ \mathbf m[a]$ is a one-one function.

The function $d$ is an admissible partial function from
$\mathcal W$ to $T$ because, $\epsilon$
being in $\mathbf E^\CatFil$,
$\rr(e)\in\mathcal W$. $d$ is local, because if
$G\in\mathcal G$, there is an $F\in\mathcal F$ such that
$a(F)\subseteq G$, and then
$d(\mathbf{\tilde e}[e](F))\subseteq G$; however,
$\mathbf{\tilde e}[e](F)\in\mathcal W$ because
$\epsilon\in\mathbf E^\CatFil$.

Let $\delta\in\CatFil(\mathcal W,\mathcal G)$ be defined by $\delta=d/\mathcal W$. Since $e$, $a$, $b$, and $m$, and $d$ are all admissible, local partial functions, the diagram of germs commutes.
The uniqueness of the diagonal arrow $\delta$ follows
by invoking either the fact that
$\mathbf E^\CatFil$ consists of epi, or the fact that
$\mathbf M^\CatFil$ consists of monic arrows of $\CatFil$.
\end{proof}

\subsection{\texorpdfstring{Epi and monic arrows in $\CatFil$}{EpiMonicCatFil}}
\begin{theorem}\mylabel{T:CatFilFacts} Let $\varphi\in\CatFil(\mathcal F,\mathcal G)$, where $\mathcal F\in\Fil S$ and $\mathcal G\in\Fil T$. We have
\begin{enumerate}
\item If $\varphi$ is epi, then $\varphi\in\mathbf E^\CatFil$; and
\item if $\varphi$ is monic, then $\varphi\in\mathbf M^\CatFil$.
\end{enumerate}
\end{theorem}

\begin{proof}
(1):
Suppose that $\varphi\in\CatFil(\mathcal F,\mathcal G)$ but $\varphi\notin\mathbf E^{\CatFil}(\mathcal F,\mathcal G)$, i.e. that $\varphi(\mathcal F)<\mathcal G$. This means that if $\varphi=f/\mathcal F$, there is an $F\in\mathcal F$ such that $F\subseteq\dd(f)$ and $f(F)\notin\mathcal G$; thus, for all $G\in\mathcal G$, $G\not\subseteq f(F)$, so there is an element $g_G\in G$ such that $g_G\notin f(F)$. Let $W=\{\,0,1\,\}$. Let $a:T\to W$ send all elements to $0$. Let $b:T\to W$ be the same, except for elements of the form $g_G$ (for any $G$), which it should send to $1$. $a$ and $b$ are total functions, so admissible. By construction, $a/\mathcal G\not= b/\mathcal G$. But, we have $a\circ f=b\circ f$, because if $x\in F$, we cannot have $f(x)=g_G$ for any $G$. Thus, $\varphi$ is not epi.

(2):
Suppose $\varphi\in\CatFil(\mathcal F,\mathcal G)$, but $\varphi\not\in\mathbf M^\CatFil(\mathcal F,\mathcal G)$. That is, we assume that if $\varphi=f/\mathcal F$, for an admissible partial function $f:\mathcal F\to T$, then $f$ is not one-one. For every $F\in\mathcal F$, such that $F\subseteq\dd(f)$, there are $a_F$, $b_F\in F$ such that $a_F\not= b_F$ but $f(a_F)=f(b_F)$. Let $W$ be the set of $F\in\mathcal F$ such that $F\subseteq\dd(f)$, and define $a:W\to S$, $b:W\to S$ by $a:F\mapsto a_F$ and $b:F\mapsto b_F$. Let $\mathcal H\in\Fil W$ be defined as $\mathcal H = a^{-1}(\mathcal F)\wedge b^{-1}(\mathcal F)$.
The functions $a$ and $b$ are total functions, hence admissible. By monotonicity, we have
\begin{align*}
a(\mathcal H)&=a(a^{-1}(\mathcal F)\wedge b^{-1}(\mathcal F))\\
&\leq a(a^{-1}(\mathcal F))\\
&\leq\mathcal F
\end{align*}
and similarly, $b(\mathcal H)\leq\mathcal F$.
Thus, $a/\mathcal H=\alpha$ and $b/\mathcal H=\beta$ are local germs, and we have $\varphi\circ\alpha=\varphi\circ\beta$ because $f\circ a=f\circ b$. However, $\alpha\not=\beta$, for, if $H\in\mathcal H$, then there exist $F_a$, $F_b\in\mathcal F$ such that $a^{-1}(F_a)\cap b^{-1}(F_b)\subseteq H$, and letting $F=F_a\cap F_b$, we have
$a^{-1}(F)\cap b^{-1}(F)\subseteq H$. Then $a_F$ and $b_F\in H$, so $a|_H\not= b|_H$. This proves $\alpha\not=\beta$, and the contrapositive, that $\varphi$ is not monic.
\end{proof}

\subsection{Isomorphisms}
As usual with factorization systems, the isomorphisms in
$\CatFil$ are precisely those arrows contained both in
$\mathbf E^\CatFil$ and in $\mathbf M^\CatFil$. (This follows from
axioms \cite[Section~\zref{I-S:FactSys}, (F3) and~(F4)]{reltgeneq}.) This allows us to characterize
them:

\begin{proposition}
An arrow $\varphi\in\CatFil(\mathcal F,\mathcal
G)$ is an isomorphism in $\CatFil$ iff there is a
partial function $f$ representing $\varphi$ such
that $f$ is one-one and $f(\mathcal F)=\mathcal G$.
\end{proposition}

\subsection{\texorpdfstring{The partially-ordered sets $\mathcal F/\mathcal M^{\CatFil}$}{MPosetsCatFil}}
Recall \cite[Section~\zref{I-S:FactSys}.\zref{I-S:PosetDefs}]{reltgeneq}  that if we have a factorization system $\pair{\mathbf E}{\mathbf M}$ in a category $\mathbf C$, $c\in\mathbf C$, and arrows $m$ and $m'$ with common codomain $c$, then we say $m\leq m'$ when there is a  diagram
\[\xymatrix{\dom m\ar[ddr]_m \ar@{.>}[rr]^f &&\dom m' \ar[ddl]^{m'.}\\
\\
&c}\]
where $f$ is an arrow making the diagram commutative.  If $f$ is an isomorphism, so that $m\leq m'$ and $m'\leq m$, then we say that $m$ and $m'$ are \emph{equivalent}, or $m\sim m'$. The $\leq$ relation defines a preorder and the $\sim$ relation defines an equivalence relation; we denote the corresponding partially-ordered set of equivalence classes (wrt $\sim$) by $c/\mathbf M$.

\begin{theorem}\mylabel{T:MSubLattice} Let $\mathcal F\in\CatFil$. Then
$\mathcal F/\mathbf M^{\CatFil}$ is isomorphic to the complete lattice $\Fil\mathcal F$.
\end{theorem}

\begin{proof} The elements of $\mathcal F/\mathbf M^\CatFil$ are equivalence classes of arrows of $\CatFil$ with codomain $\mathcal F$.  Given a germ $\mu:\mathcal G\to\mathcal F$, we map $\mu$ to $\mu(\mathcal G)\in\Fil\mathcal F$.
If we have a diagram
\begin{equation}\mylabel{Diag:SimDiag}\xymatrix{\mathcal G\ar[ddr]_\mu \ar@{.>}[rr]^\varphi &&\mathcal G' \ar[ddl]^{\mu'.}\\
\\
&\mathcal F}\end{equation}
where $\varphi$ is an isomorphism, 
then because $\mu=\mu'\circ\varphi$, $\mu(\mathcal G)=\mu'(\varphi(\mathcal G))$. Now, $\varphi$, being an isomorphism, is an arrow of $\mathbf E^{\CatFil}$, and by definition, this means that $\varphi(\mathcal G)=\mathcal G'$. Thus, $\mu(\mathcal G)=\mu'(\mathcal G')$. In other words, $\mu\sim\mu'$ implies $\mu(\mathcal G)=\mu'(\mathcal G')$. Thus we have defined a mapping $Z:\mathcal F/\mathbf M^{\CatFil}\to\Fil F$, which takes the equivalence class $[\mu]$ of an arrow $\mu:\mathcal G\to\mathcal F$ to $\mu(\mathcal G)$.

Suppose now that we have the diagram \ref{Diag:SimDiag}, absent  $\varphi$, but knowing that $\mu(\mathcal G)\leq\mu'(\mathcal G')$ (in the lattice $\Fil\mathcal F$), and we will construct an arrow $\varphi$ witnessing $\mu\leq \mu'$. Let $m:G\to S$, $m':G'\to S$ be admissible one-one partial functions representing the germs $\mu$ and $\mu'$, respectively.

Let $f$ be the partial function $(m')^{-1}\circ m$. We will show that $f$ is a local, admissible partial function, such that $\varphi=f/\mathcal G$ completes Diagram~\ref{Diag:SimDiag}.

Let $K'\in\mathcal G'$. We have $K'\cap G'\in\mathcal G'$ since $G'\in\mathcal G'$. Then $m'(K'\cap G')\in m'(\mathcal G')$. Since $m(\mathcal G)\leq m'(\mathcal G')$ by assumption, there is an $L\in m(\mathcal G)$ such that $L\subseteq m'(K'\cap G')$. There is a $K\in\mathcal G$ such that $K\subseteq G=\dd(m)$  and $K\subseteq L$. Then, $K\subseteq \dd(f)$ and $f(K)\subseteq K'$. So, $f$ is local.

Certainly, $m|_K=m'\circ f|_K$, implying that $\mu=\mu'\circ\varphi$.
\end{proof}

\section{\texorpdfstring{Properties of the Category $\CatFil$}{PropertiesCatFil}}\mylabel{S:CatFilProps}

In this section, we will verify that the category $\CatFil$ satisfies the basic properties in the list \cite[\zref{I-A:Basic}]{reltgeneq}. (A number of subsequent theorems in \cite{reltgeneq} will follow.)

\begin{theorem} $\mathbf M^\CatFil$ is well-powered.
\end{theorem}

\begin{proof} If $\mathcal F\in\Fil S$, then by Theorem~\ref{T:MSubLattice}, $\mathcal F/\mathbf M^\CatFil\cong\Fil\mathcal F$, which is a small set because $S$ is a small set.
\end{proof}

\begin{theorem}\mylabel{T:FilHasFiniteLimits} $\CatFil$ has limits of all finite diagrams.
\end{theorem}

\begin{proof} It suffices to show that there are equalizers, and products of finite tuples of objects.

(Equalizers): Suppose $\alpha$, $\beta:\mathcal F\to\mathcal G$, where $\mathcal F\in\Fil S$ and $\mathcal G\in\Fil T$.  Let $\pair{\mathcal H}\mu$ be the equalizer and its arrow to $\mathcal F$, if they exist, in the diagram
\[\xymatrix{\mathcal H\ar[r]^\mu&\mathcal F\ar@<.5ex>[r]^\alpha
\ar@<-.5ex>[r]_\beta&\mathcal G,}\]
where we know that $\mu$ needs to be monic because that is always the case for an equalizer, and
from our analysis of the factorization system $\pair{\mathbf E^\CatFil}{\mathbf M^\CatFil}$, that if the equalizer exists, we can choose $\mathcal H\in\Fil\mathcal F$ and $\mu$ the germ with respect to $\mathcal H$ of the identity function $1_S$. Also, we know that if $\nullset\in\mathcal G$, then $S=\nullset$ and $\mu$ is an isomorphism, but if not, then by Remark~\ref{R:TotalFunction}, there are total functions $m$, $a$, and $b$ representing $\mu$, $\alpha$, and $\beta$ respectively.

We choose
\[\mathcal H=\Fg\{\,[a=b]\mid a/\mathcal F=\alpha\text{ and }b/\mathcal F=\beta\,\},\]
where $[a=b]$ stands for the set of $x\in S$ on which the partial functions $a$ and $b$ are defined and $a(x)=b(x)$. We have $\mathcal H\leq\mathcal F$, because if $F\in\mathcal F$, $a/\mathcal F=\alpha$, and $b/\mathcal F=\beta$, then $[a|_F=b|_F]\subseteq F$.

Suppose now that $\mathcal K$ is a filter, and $\gamma\in\CatFil(\mathcal K,\mathcal F)$ is such that $\alpha\circ\gamma=\beta\circ\gamma$. Let $a$, $b$, and $g$ be admissible partial functions representing $\alpha$, $\beta$, and $\gamma$ respectively.  We have $[a\circ g =b\circ g]=K\in\mathcal K$. Let $h:K\to S$ be defined as $h:x\mapsto g(x)$. Clearly $\alpha\circ (h/\mathcal K)=\beta\circ(h/\mathcal K)$.
We must show that $h/\mathcal K$ is local. If $F\in\mathcal F$, then $F\supseteq [a|_F=b|_F]\in\mathcal H$. Let $K'=[(a|_F)g=(b|_F)g]$; we have $K'\in\mathcal K$ and $g(K')\subseteq F$.

The germ $h/\mathcal K$ is unique, because, $\mu$ being one-one on its domain of definition, $\mu\in\mathbf M^\CatFil$, and by Theorem~\ref{T:CatFilFactSys}, then, $\mu$ is monic.

(Products of finite tuples of filters):
Given a tuple of filters $\mathcal F_i$, on sets $S_i$, let $S=\Pi_iS_i$, and let $\mathcal P=\bigwedge_i\pi_i^{-1}(\mathcal F_i)$, where the $\pi_i:S\to S_i$ are the projections. We claim that if the tuple $\mathcal F_i$ is finite, then $\mathcal P$ is the product, with product cone the germs $\pi_i/\mathcal P$.

Given a filter $\mathcal G$ (on a set $G$), and germs $\varphi_i:\mathcal G\to\mathcal F_i$, let $f_i:X_i\to S_i$ be a partial function representing $\varphi_i$, for each $i$, where $X_i\in\mathcal G$. We define $X=\bigcap_iX_i\in\mathcal G$ (using the fact that the index set is finite) and we can define $f:X\to S$, using the universal property of the product. It is clear that for each $i$, $(\pi_i/\mathcal P)\circ(f/\mathcal G)= \varphi_i$. 
\end{proof}

\begin{theorem}
$\CatFil$ has pullbacks of small tuples of arrows in $\mathbf M^\CatFil$.
\end{theorem}

\begin{proof} 
Let $\mathcal F\in\Fil S$,  and for each $i$ in some small index set which we take to be an ordinal without loss of generality (as long as we assume the Axiom of Choice), let $\mathcal G_i\in\CatFil$ and $\gamma_i\in\mathbf M^\CatFil(\mathcal G_i,\mathcal F)$. We will prove that the diagram
\begin{equation}\mylabel{Diag:PullbackDiagram}
\xymatrix{
\mathcal G_0\ar[ddr]_{\gamma_0} & \mathcal G_1\ar[dd]_{\gamma_1} & \mathcal G_2 \ar[ddl]_{\gamma_2} & \\
&& \cdots\\
&\mathcal F
}\end{equation}
 has a pullback.
 
By Theorem~\ref{T:MSubLattice} and because the lattice $\Fil\mathcal F$ is complete, the $\mathbf M^\CatFil$-subobjects $\gamma_i/\mathbf M^\CatFil$ have a meet, $\mathcal H$. Let $\eta:\mathcal H\to\mathcal F$ be the germ $1_S/\mathcal H$. By Theorem~\ref{T:MSubLattice}, we can then draw the pullback diagram
\begin{equation}
\mylabel{Diag:PullbackInSet}
\xymatrix{
& \mathcal H\ar@{.>}[ddl]_{\varphi_0} \ar@{.>}[dd]_{\varphi_1} \ar@{.>}[ddr]_{\varphi_2}\\
&&\cdots\\
\mathcal G_0\ar[ddr]_{\gamma_0} & \mathcal G_1\ar[dd]_{\gamma_1} & \mathcal G_2 \ar[ddl]_{\gamma_2} & \\
&& \cdots\\
&\bigcup\mathcal F
}\end{equation}
where for all $i$, $\gamma_i\circ\varphi_i=1_S/\mathcal H$, with the $\varphi_i$ coming from Theorem~\ref{T:MSubLattice} and the fact that for each $i$, $\mathcal H\leq\mathcal G_i$.
\end{proof}

\begin{theorem}
$\CatFil$ has coproducts of all small tuples of arrows.
\end{theorem}

\begin{proof}
Let $\mathcal F_i\in\Fil S_i$, indexed by a small set which, without loss of generality, we take to be an ordinal number.
Let $Z$ be the disjoint union of the sets $S_i$, and for each $i$, let $j_i$ be the insertion of $S_i$ into $Z$. Then, let $\mathcal \bigvee_i j_i(\mathcal F_i)$ be the  join over $i$ of the filters $j_i(\mathcal F_i)$. The join $\mathcal C=\bigvee_i j_i(\mathcal F)$ is the filter of subsets $C\subseteq Z$ such that for all $i$, there is an $F_i\in\mathcal F_i$ with $j_i(F_i)\subseteq C$. Let $\iota_i= j_i/\mathcal F_i$ for each $i$.

Thus we have the diagram
\[\xymatrix{
\mathcal F_0\ar[ddr]_{\iota_0} & \mathcal F_1\ar[dd]_{\iota_1} & \mathcal F_2 \ar[ddl]_{\iota_2} & \\
&& \cdots;\\
&\mathcal C
}\]
in the category $\CatFil$, and we will show that $\mathcal C$ is a coproduct of the tuple of $\mathcal F_i$.

Suppose we are given a filter $\mathcal X\in\Fil W$, and a cocone of germs $\xi_i:\mathcal F_i\to\mathcal X$, as shown in the following diagram, and we will construct and prove uniqueness of the dotted arrow $\lambda$
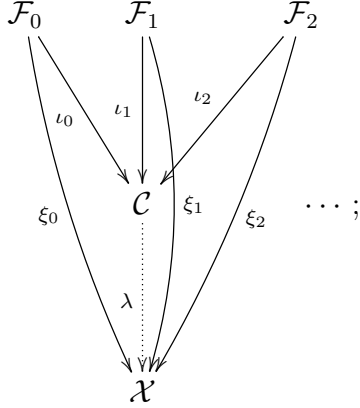
\[\xymatrix{
\mathcal F_0\ar[ddr]_{\iota_0}\ar@/_/[ddddr]_{\xi_0} & \mathcal F_1\ar[dd]_{\iota_1}\ar@/^12pt/[dddd]^{\xi_1} & \mathcal F_2 \ar[ddl]_{\iota_2}\ar@/^/[ddddl]^{\xi_2} \\
&& \\
&\mathcal C\ar@{.>}[dd]_\lambda & \qquad\cdots;\\
\\
&\mathcal X
}\]
such that for all $i$, $\xi_i=\lambda\circ\iota_i$, proving the universal property.

For each $i$, let $x_i:F_i\to W$ be a partial function representing $\xi_i$. Let $C=\bigcup_i F_i\subseteq Z$. Let $\ell:C\to W$ be the arrow given by the universal property of the disjoint union (in $\Set$). Then let $\lambda=\ell/\mathcal C$.
$\ell$ is an admissible partial function; we must show that it is local, that its germ $\lambda$ satisfies $\xi_i=\lambda\circ\iota_i$ for all $i$, and that if $\lambda'$ is any germ satisfying those equations, $\lambda'=\lambda$.

Each partial function $x_i$ is local, which means that if $X\in\mathcal X$, there is an $F'_i\in\mathcal F_i$ such that $x_i(F'_i)\subseteq X$. It follows that $\ell'$, defined as the restriction of $\ell$ to $C'=\bigcup_i j_i(F'_i)$, is local.

Now suppose that we have any arrow $\lambda':\mathcal C\to\mathcal X$ such that for all $i$, $\xi_i=\lambda'\circ\iota_i$. Let $\ell':C'\to W$ be an admissible, local partial function representing $\lambda'$. For each $i$, let $F''_i$ be the subset of $S_i$ such that $\ell_i\circ j_i=\ell'\circ j_i$ on $F''_i$.  We know that $F''_i\in\mathcal F_i$ because $\lambda\circ\iota_i=\xi_i=\lambda'\circ\iota_i$. Then $C''\in\mathcal C$ where $C''=\bigcup_i j_i(F''_i)$, and $\ell=\ell'$ on $C'$. This shows that $\lambda=\ell/\mathcal C=\ell'/\mathcal C=\lambda'$.
\end{proof}

\begin{theorem}
$\mathbf M^{\CatFil}$ consists of monic arrows of $\CatFil$.
\end{theorem}

\begin{proof}
More than that, by Theorems \ref{T:CatFilFactSys} and~\ref{T:CatFilFacts}, arrows in $\CatFil$ are monic iff they belong to $\mathbf M^\CatFil$.
\end{proof}

\begin{theorem}
$\mathbf E^\CatFil$ is stable under pullbacks along arrows of $\CatFil$.
\end{theorem}

\begin{proof}
Let $\varepsilon\in\mathbf E^\CatFil(\mathcal F,\mathcal G)$, and let us pull it back along $\varphi:\mathcal H\to\mathcal G$, giving $\mathcal P$ and $\varepsilon'\in\CatFil(\mathcal P,\mathcal H)$ which we want to prove is in $\mathbf E^\CatFil$. 

First, let us deal with the boundary case in which $\nullset\in\mathcal G$. Then also $\nullset\in\mathcal F$ and $\varepsilon$ is an isomorphism, whence the pullback $\varepsilon'$ is too, so that $\varepsilon'\in\mathbf E^\CatFil$.

Assuming, on the contrary, that $\nullset\notin\mathcal G$, we can use Remark~\ref{R:TotalFunction}, and drawing the detailed diagram
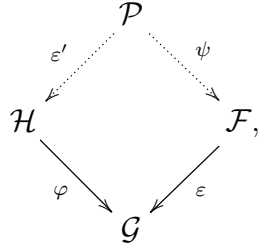
\[\xymatrix{&\mathcal P\ar@{.>}[dl]_{\varepsilon'}\ar@{.>}[dr]^\psi\\
\mathcal H\ar[dr]_\varphi &&\mathcal F,\ar[dl]^\varepsilon\\
& \mathcal G
}\]
we can assume that there are total functions $f$ and $e$ representing $\varphi$ and $\varepsilon$, respectively. We can construct our pullback $\varepsilon'$ by first constructing the product
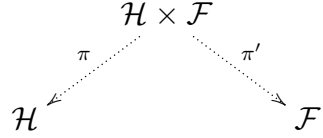
\[\xymatrix{
&\mathcal H\times\mathcal F\ar@{.>}[dl]_\pi\ar@{.>}[dr]^{\pi'}\\
\mathcal H && \mathcal F
}\]
and then the equalizer of the arrows $\varphi\circ\pi$, $\varepsilon\circ\pi':\mathcal H\times\mathcal F\to\mathcal G$:
\[\xymatrix{
\mathcal P\ar@{.>}[rr]^\lambda&&
\mathcal H\times\mathcal F\ar@<-.6ex>[rr]_{\varphi\circ\pi}\ar@<.6ex>[rr]^{\varepsilon\circ\pi'}
&&\mathcal G,
}\]
after which we can set $\varepsilon'=\pi\circ\lambda$ and $\psi=\pi'\circ\lambda$. 

To show $\varepsilon'\in\mathbf E^\CatFil$, we need to show that $\mathcal H=\varepsilon'(\mathcal P)$, or in other words, since we have $\varepsilon'(\mathcal P)\leq\mathcal H$ just because $\varepsilon'$ is an arrow, that $\mathcal H\leq\varepsilon'(\mathcal P)$. To show this, it will suffice to show that for some admissible partial function $e'$ representing $\varepsilon'$, and any $P\in\mathcal P$, there is an $H\in\mathcal H$ such that $H\subseteq e'(P)$.

Examining the proof of the existence of equalizers, we see that under our current assumption that $\nullset\notin\mathcal G$, not only $\varphi$ and $\varepsilon$, but also $\lambda$ can be represented by a total function, and so can $\pi$ and $\pi'$ from the proof of the existence of finite products. Thus, the partial function $e'$ representing $\varepsilon'$ can be assumed to be a total function.

Let $P\in\mathcal P$. Let $F\in\mathcal F$, and $H\in\mathcal H$, be such that $f(H)\subseteq e(F)$. (This is possible because $\varepsilon\in\mathbf E^\CatFil$ and $f$ is local.)  At the same time, let $H$ and $F$ be such that  $\pi^{-1}(H)\cap(\pi')^{-1}(F)\subseteq P$.  Consider the diagram in $\Set$, which is a pullback diagram:
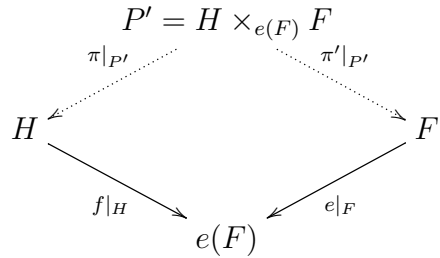
\[\xymatrix{&P'=H \times_{e(F)}  F\ar@{.>}[dl]_{\pi|_{P'}}\ar@{.>}[dr]^{\pi'|_{P'}}\\
H\ar[dr]_{f|_H} && F\ar[dl]^{e|_F}\\
& e(F)
}\]
where we know very well that in $\Set$, since $e$ maps $F$ onto $e(F)$, the function $\pi$ also maps $P'$ onto $H$. But $P'\subseteq P$. Since $H$ can be made as small as desired in the filter $\mathcal H$, this proves that $\varepsilon'=\pi/\mathcal P\in\mathbf E^\CatFil$.

\end{proof}

\section{The Core of a Filter}\mylabel{S:Core}

\begin{definition} Let $\mathcal F$ be a filter of subsets of $S$. The subset $\bigcap_{F\in\mathcal F}F\subseteq S$ is called the \emph{core} of $\mathcal F$ and denoted by $\core\mathcal F$. If $\varphi\in\CatFil(\mathcal F,\mathcal G)$, then $\core \varphi$ will denote the restriction of $f$ to the filter $\Fg\{\,\core\mathcal F\,\}\leq\mathcal F$.
\end{definition}

\begin{proposition} 
The mapping $S\mapsto \{\,S\,\}$, and the functor $\Gamma$ sending $f:S\to S'$ to $f/\{\,S\,\}$, define a functor $L:\Set\to\LPartial$.
There are adjunctions
\[\langle L,\core,\eta,\varepsilon\rangle:\Set\rightharpoonup\LPartial\]
where $\eta_S=1_S$ and $\varepsilon_{\pair S{\mathcal F}}$ is the inclusion of $\core\mathcal F$ into $S$,
and
\[\langle \Gamma\circ L,\core,\alpha'\rangle:\Set\rightharpoonup\CatFil,\]
where $\Gamma:\LPartial\to\CatFil$ is the functor defined in Subsection~\ref{S:CatFilSub}.
\end{proposition}

\begin{proof} To have these adjunctions, we must have isomorphisms
\[\chi_{S,\mathcal G}:\LPartial(L(S),\mathcal G)\cong\Set(S,\core\mathcal G)\]
and
\[\chi'_{S,\mathcal G}:\CatFil(\Gamma(L(S)),\mathcal G)\cong\Set(S,\core\mathcal G)\]
natural in $S$ and $\mathcal G$.

Since $L(S)=\{\,S\,\}$, for $L(S)$ be admissible, a partial function between values of the functor $L$ must be a total function, and the functor $\Gamma$ does nothing.  Thus, the isomorphisms $\chi$ and $\chi'$ simply relate total functions to total functions.  Naturality in $S$ and in $\mathcal G$ is straightforward.
\end{proof}

\begin{remark} Thus, if we decide to study objects in $\CatFil$ that have additional structure, such as groups or other algebra structures, we have available right adjoint \emph{forgetful} functors that will yield groups or other algebras.  However, note that the forgetting that the core functor does can be very extensive. For example, a filter can have an empty core, or a core with just one element. Certainly we should expect to have more of interest to study in many such cases than the trivial group.
\end{remark}

\section{Monoidal Products}\mylabel{S:FilMonoidal}

\subsection{\texorpdfstring{$\mathcal F\opsquare\mathcal G$}{OpsquareCatFil}}\mylabel{S:FilOpSquare}
Suppose that $\mathcal F$ and $\mathcal G$ are filters of subsets of sets $S$ and $T$, respectively. If
$F\in\mathcal F$ and $g:F\to\mathcal G$ (i.e., $g$
is any function assigning a subset in the filter $\mathcal
G$ to each element of $F$), then we define
$F\opsquare g=\{\,\pair st\mid s\in F,t\in
g(s)\,\}$. More generally, if $g:F\to\Sub S$ is any function such that $\{\,s\mid g(s)\in\mathcal G\,\}\in\mathcal F$, then we define $F\opsquare g$ just the same, as $\{\,\pair st\mid s\in F,t\in g(s)\,\}$.

\begin{theorem} We have
\begin{enumerate}
\item 
The set \[\{\,F\opsquare
g\mid F\in\mathcal
F,\,g:F\to\mathcal G\,\}\] 
 is a base
for a filter $\mathcal F\opsquare\mathcal G$ of subsets of
$S\times T$;
\item
the filter $\mathcal F\opsquare\mathcal G$ consists of those subsets $H\subseteq S\times T$ such that \[\{\,s\mid\{\,t\in T\mid\pair st\in H\,\}\in \mathcal G\,\}\in\mathcal F;\]
\item every subset in $\mathcal F\opsquare\mathcal G$ has the form $S\opsquare g$ for some $g:S\to\Sub T$.
\end{enumerate}
\end{theorem}

\begin{proof}
(1): Given $F_1\opsquare g_1$ and $F_2\opsquare g_2$, we have $F\opsquare g\subseteq (F_1\opsquare g_1)\cap(F_2\opsquare g_2)$, where $F=F_1\cap F_2$ and for $s\in F$, 
$g(s)=
g_1(s)\cap g_2(s)
$. 

(2): Given $H$ satisfying the condition, let $F=\{\,s\mid\{\,t\in T\mid\pair st\in H\,\}\in\mathcal G\,\}\in\mathcal F$. For each $s\in F$, let $g(s)=\{\,t\mid\pair st\in H\,
\}\in\mathcal G$. Then $F\opsquare f\subseteq H$, so $H\in\mathcal F\opsquare\mathcal G$.

On the other hand, if $D\opsquare f$ is an element of the base of $\mathcal F\opsquare\mathcal G$, then it clearly satisfies the condition. Then we need only see that the set of subsets satisfying the condition is closed upward.

(3): Given $H\in\mathcal F\opsquare\mathcal G$, let $g(s)=\{\,t\in T\mid\pair st\in H\,\}$. Then $H=S\opsquare g$.

\end{proof}

Part (2) of the Theorem suggests some notation we will use later:

\begin{notation}\label{N:Fh}
	If $\mathcal F$ is a filter of subsets of a set $S$, $\mathcal G$ is filter of subsets of a set $T$, and we have a subset $X\in\mathcal F\opsquare\mathcal G$, then we define
	\begin{enumerate}
		\item  $F_{\mathcal F,\mathcal G,X}=\{\,s\in S\mid\{\,t\in T\mid\pair st\in X\,\}\in\mathcal G\,\}$ and
		\item  $h_{\mathcal F,\mathcal G,X}=[\,s\in F_{\mathcal F,\mathcal G,X}\mapsto\{\,t\in T\mid\pair st\in X\,\}\,]$.
	\end{enumerate}
\end{notation}

\subsection{\texorpdfstring{Monoidal products in $\LPartial$ and $\CatFil$}{OpsquareLPartialCatFil}}
\begin{definition}
If $f\in\LPartial(\mathcal
F,\mathcal F')$ and $g\in\LPartial(\mathcal G,\mathcal G')$, where
$\dd(f)=F$ and $\dd(g)=G$, then we define
$f\opsquare_pg:\mathcal F\opsquare\mathcal G\to\mathcal F'\opsquare\mathcal G'$ to be
the partial function with domain of definition
$F\times G$, sending a pair
$\pair st$ to
$\pair{f(s)}{g(t)}$.
\end{definition}

\begin{theorem} We have
\begin{enumerate}
\item The foregoing defines a functor
$\opsquare_p:\LPartial\times\LPartial\to\LPartial$.
\item If $f\equiv_{\mathcal F} f'$ and $g\equiv_{\mathcal G}g'$, then
$(f\opsquare_pg)\equiv_{\mathcal F\opsquare\mathcal G}
(f'\opsquare_pg')$.
\item Setting
$\opsquare_g=\opsquare_p$
on objects, and
$(f/{\mathcal F})\opsquare_g(g/{\mathcal G})=(f\opsquare_pg)/{\mathcal F\opsquare\mathcal G}$ on arrows, defines a functor
$\opsquare_g:\CatFil\times\CatFil\to\CatFil$.
\item $\Gamma(-_1\opsquare_p-_2)
=\Gamma(-_1)\opsquare_g\Gamma(-_2):\LPartial\times
\LPartial\to\CatFil$.
\item $\core(-_1\opsquare_p-_2)=
\core(-_1)\times\core(-_2):\LPartial\times\LPartial\to
\Set$.
\item $\core(-_1\opsquare_g-_2)=
\core(-_1)\times\core(-_2):\CatFil\times\CatFil\to\Set$.
\end{enumerate}
\end{theorem}

\begin{proof} (1): Let $f\in\LPartial(\mathcal F,\mathcal F')$, and  $g\in\LPartial(\mathcal G,\mathcal G')$. $F\times G\in\mathcal F\opsquare \mathcal G$, because $F\times G=F\opsquare [s\in F\mapsto G]$. Thus, $f\opsquare g$ is admissible. To show $f\opsquare g$ is local, consider a base element $F'\opsquare h'\in \mathcal F'\opsquare \mathcal G'$, where $F'\in\mathcal F'$ and $h':F'\to\mathcal G'$. We have $f^{-1}(F')\in\mathcal F$. Consider now $f^{-1}(F')\opsquare h$, where $h:f^{-1}(F')\to\mathcal G$ is defined by setting $h(s)=g^{-1}(h'(f(s)))$. Then if $\pair st\in f^{-1}(F')\opsquare h$, we have 
\begin{align*}
(f\opsquare g)\pair st&=\pair{f(s)}{g(t)}\in F'\opsquare h',
\end{align*}
because $f(s)\in F'$ and $g(t)\in h'(f(s))$.

(2): Suppose $f|_F=f'|_F$ and $g|_G=g'|_G$, where $F\in\mathcal F$ and $G\in\mathcal G$. Then if $\pair st\in F\times G $, we have $(f\opsquare_p g)\pair st=(f'\opsquare_p g')\pair st$. But $F\times G\in\mathcal F\opsquare\mathcal G$, so $(f\opsquare_p g)\equiv_{\mathcal F\opsquare_p\mathcal G}(f'\opsquare g')$.

(3): Follows from (2).

(4): Follows by the definition of $\opsquare_g$ in (3).

(5), (6):  The core functor is a right adjoint functor, after all.
\end{proof}

\subsection{Unit object, unit and associativity natural
isomorphisms, and coherence}\mylabel{S:FilMonoidalStuff}

If $S$, $T$, and $W$ are sets, let
$\alpha_{S,T,W}:S\times (T\times W)\to
(S\times T)\times W$ be the function
defined by
$\pair s{\pair tw}\to\pair{\pair st}w$.
If $\mathcal D$, $\mathcal D'$, and
$\mathcal D''$ are filters on $S$, $T$, and $W$ respectively,
we set
$\alpha^p_{\mathcal D,\mathcal D',\mathcal D''}=\alpha_{S,T,W}$,
considered as a partial
function. We set $\alpha^g_{\mathcal D,\mathcal D',\mathcal D''}=\alpha_{S,T,W}/(\mathcal D\opsquare(\mathcal D'\opsquare\mathcal D''))$. 

\begin{theorem}
$\alpha^p_{\mathcal D,\mathcal D',\mathcal D''}\in\LPartial(\mathcal D\opsquare(\mathcal D'\opsquare\mathcal D''),(\mathcal D\opsquare\mathcal D')\opsquare\mathcal D'')$ and is an isomorphism.
Similarly, its germ in $\CatFil$ is an isomorphism.
\end{theorem}

\begin{proof} $\alpha^p_{\mathcal D,\mathcal D',\mathcal D''}$ is admissible because $\alpha_{S,T,W}$ is a total function. To show it is an isomorphism, it suffices to show that both it and its inverse (in $\Set$) are local.

If $X\in\mathcal D\opsquare(\mathcal D'\opsquare\mathcal D'')$, then $X=S\opsquare f$ where $f:S\to\Sub (T\times W)$. In this proof, we will denote the subset $X=S\opsquare f$ of $S\times(T\times W)$ by $X[f]$.

On the other hand, if $Y\in(\mathcal D\opsquare\mathcal D')\opsquare\mathcal D''$, then $Y=Y[h,k]=(S\opsquare h)\opsquare k$ where $h:S\to\Sub T$ and $k:S\times T\to\Sub W$.

To see that $\alpha_{S,T,W}^{-1}((\mathcal D\opsquare\mathcal D')\opsquare\mathcal D'')= \mathcal D\opsquare(\mathcal D'\opsquare\mathcal D'')$, consider $X[f]$, and we will find $h$ and $k$ such that $\alpha_{S,T,W}^{-1}(Y[h,k])= X[f]$. Let $h:S\to\Sub T$ be defined by $h:s\mapsto p_1(f(s))$, where $p_1:T\opsquare g\to T$ is the projection to the first component of a pair. Let $k:S\opsquare h\to\Sub W$ be defined by $k:\pair st\mapsto p_2(f(s))$, where $p_2:T\opsquare g\to\Sub W$ projects a pair to its second component.
We have
\begin{align*}
\pair s{\pair tw}\in\alpha_{S,T,W}^{-1}(Y[h,k])&\iff\alpha_{S,T,W}\pair s{\pair tw}\in Y[h,k]\\
&\iff\pair{\pair st}w\in Y[h,k]\\
&\iff s\in S\text{ and }t\in h(s)\text{ and }w\in k(s,t)\\
&\iff s\in S\text{ and }\pair tw\in h(s)\opsquare k(s,t)\\
&\iff s\in S\text{ and } \pair tw\in f(s)\\
&\iff \pair s{\pair tw}\in X[f];
\end{align*}
it follows that both $\alpha_{\mathcal D,\mathcal D',\mathcal D''}$ and $\alpha_{\mathcal D,\mathcal D',\mathcal D''}^{-1}$ are local.

Applying the functor $\Gamma$, we obtain the corresponding statements about $\alpha^g_{\mathcal D,\mathcal D',\mathcal D''}$.
\end{proof}

\begin{notation}\mylabel{N:unit} Let $u$ denote the filter $\{\,1\,\}=\{\,\{\,0\,\}\,\}$ on the one-element set $1=\{\,0\,\}$.
\end{notation}

 If $\mathcal D$ is
a filter on the set $S$, and
$\lambda_S:1\times S\to S$ is the function defined by
$\pair 0s\mapsto s$, then we define $\lambda^p_{\mathcal D}=\lambda_S$, considered as a partial
function from $u\opsquare_p
\mathcal D$ to
$\mathcal D$. We also define
$\lambda^g_{\mathcal D}=\lambda^p_{\mathcal D}/{\mathcal D\opsquare u}\in\CatFil(u\opsquare\mathcal D,\mathcal D)$.
Similarly, if $\varrho_S:S\times 1\to S$ is the function
defined by $\pair s0\mapsto s$, then we  define
$\varrho^p_{\pair S{\mathcal D}}=\varrho_S$, considered
as a partial function from $\mathcal D
\opsquare u$ to
$\mathcal D$, and $\varrho^g_{\mathcal D}=\varrho^p_{\mathcal D}/{\mathcal D\opsquare u}\in\CatFil(\mathcal D\opsquare u,\mathcal D)$.

Note that $u$ is a terminal object both in
$\LPartial$ and in $\CatFil$ and that for any filter $\mathcal F$, the set $\core\mathcal F$ is naturally isomorphic both to $\LPartial(u,\mathcal F)$ and to $\CatFil(u,\mathcal F)$.

\renewcommand{\labelitemi}{$\text{ }$}
\begin{proposition}
These definitions yield natural isomorphisms
\begin{itemize}
\item
$\alpha^p:-_1\opsquare_p(-_2\opsquare_p-_3)
\cong(-_1\opsquare_p-_2)\opsquare_p-_3$,
\item
$\alpha^g:-_1\opsquare_g(-_2\opsquare_g-_3)
\cong(-_1\opsquare_g-_2)\opsquare_g-_3$,
\item $\lambda^p: u\opsquare_p-\cong -$,
\item $\lambda^g:u\opsquare_g-\cong -$,
\item $\varrho^p:-\opsquare_pu\cong -$, and
\item
$\varrho^g:-\opsquare_gu\cong -$
\end{itemize}
making $\langle\LPartial,\opsquare_p,u,\alpha^p,
\lambda^p,\varrho^p\rangle$ and
$\langle\CatFil,\opsquare_g,u,\alpha^g,\lambda^g,\varrho^g\rangle$ into
monoidal categories \cite[VII.1]{macl2}, \cite[Definition~\zref{II-D:Monoidal}]{reltclosed}, and the functors
$\Gamma:\LPartial\to\CatFil$,
$\core^p:\LPartial\to\Set$, and $\core:\CatFil\to\Set$
are strict morphisms of monoidal categories  \cite[VII.1]{macl2},\cite[Definition~\zref{II-D:StrictMonoidal}]{reltclosed}.
\end{proposition}

\section{\texorpdfstring{The Closed Category Structure of $\CatFil$}{ClosedCatFil}}\mylabel{S:FilClosed}

Let $\mathcal G$ be a filter of subsets of a set $T$,  $\mathcal
H$ be a filter of subsets of another set $W$, $q\in\LPartial(\mathcal G,\mathcal G')$, and $r\in\LPartial(\mathcal H',\mathcal H)$.
Since $\CatFil$ is a category, composition with $\gamma=q/\mathcal G$ on the left induces a function \[\CatFil(\mathcal H,\gamma):\CatFil(\mathcal H,\mathcal G)\to\CatFil(\mathcal H',\mathcal G)\]
and composition with $\gamma'=r/\mathcal H$ on the right produces a function
\[\CatFil(\gamma',\mathcal G):\CatFil(\mathcal H,\mathcal G)\to\CatFil(\mathcal H',\mathcal G);\]
and we have also
\[\CatFil(\gamma',\gamma)=\CatFil(\gamma',\mathcal G)\circ\CatFil(\mathcal H,\gamma)=\CatFil(\mathcal H,\gamma)\circ\CatFil(\gamma',\mathcal G):\CatFil(\mathcal H,\mathcal G)\to\CatFil(\mathcal H',\mathcal G'),\]
where hopefully, the reader will recognize easily that these functions are simply forms of the Hom functor for the category $\CatFil$.  We mention them to clarify our notation in what follows.

\subsection{\texorpdfstring{The internal Hom functor of $\CatFil$}{InternalHomCatFil}}
We will be defining the internal Hom functor for $\CatFil$ using, among other things, the mapping that takes an admissible (but not necessarily local) partial function to its germ. That is, if we have filters $\mathcal H$ (of subsets of $W$) and $\mathcal G$ (of subsets of $T$) then we can take the germ of an element of $\Partial(\mathcal H,\mathcal G)$ (See Notation~\ref{N:Partial}, and Section~\ref{S:Germs} where germs are defined), giving an arrow in $\CatFil(\mathcal H,\mathcal G)$. However, formation of germs is useful more generally:
If $G\in\mathcal G$, then recall that we denote the set of admissible partial functions $f:W\to T$, such that there is an $H\in\mathcal H$ such that $f(H)\subseteq G$, by $\Partial(\mathcal H,\mathcal G,G)$. If $f$ is such a function, we can form the germ $f/\mathcal H$ and get an element of the set of germs of elements of $\Partial(\mathcal H,\mathcal G,G)$, which set we can denote by $\Gamma(\Partial(\mathcal H,\mathcal G,G))$ as mentioned in Definition~\ref{D:GAPF}.

\begin{proposition}\label{T:FilInternal} We have
\begin{enumerate}
\item
The subsets $\CatFil(\mathcal H,\mathcal G,G)\subseteq\CatFil(\mathcal H,\mathcal G,T)$, for $G\in\mathcal G$, form a base for a filter $\mathcal G^{\mathcal H}$ of subsets of $\CatFil(\mathcal H,\mathcal G,T)$, and $\core \mathcal G^{\mathcal H}=\CatFil(\mathcal H,\mathcal G)$.
 
\item
$(-_1)^{(-_2)}$ is a functor from
$\CatFil\times\CatFil^{\text{op}}$ to $\CatFil$.
\end{enumerate}
\end{proposition}

\begin{proof}
(1): clear.

(2): Given $\mathcal G$ and $\mathcal H$, other filters $\mathcal G'$ and $\mathcal H'$, and germs $\gamma\in\CatFil(\mathcal G,\mathcal G')$, $\rho\in\CatFil(\mathcal H',\mathcal H)$, we must show that $\gamma$ gives rise to an arrow $\gamma^{\mathcal H}\in\CatFil(\mathcal G^{\mathcal H},{\mathcal G'}^{\mathcal H})$, and $\rho$ gives rise to an arrow $\mathcal G^\rho:\CatFil(\mathcal G^{\mathcal H},\mathcal G^{\mathcal H'})$.

Let $g$ be a partial function, admissible with respect to $\mathcal G$ and local with respect to $\mathcal G'$.  (Every $\gamma$ has such a representative by definition of $\CatFil(\mathcal G,\mathcal G')$, so let us say that $g$ is a representative of $\gamma$.) Composition with $g$ on the left is a (total) function from $\Partial(\mathcal H,\mathcal G,T)$ to $\Partial(\mathcal H,\mathcal G',T')$ (where $\mathcal G'$ is a filter of subsets of the set $T'$).  The germ of this total function is admissible (the germ of a total function always is); to show it is local, consider a basic set $\Partial(\mathcal H,\mathcal G',G')$ where $G'\in\mathcal G'$. If $G\in\mathcal G$ is such that $g(G)\subseteq G'$ (such a $G$ exists because $g$ is local with respect to $\mathcal G'$), then composition with $g$ on the left maps the basic set $\Partial(\mathcal H,\mathcal G,G)$ into $\Partial(\mathcal H,\mathcal G',G')$, as needed to show composition with $g$ is a local function, and so, the germ of the composition function is local. Thus the arrow $\gamma^{\mathcal H}:{\mathcal G}^{\mathcal H}\to{\mathcal G'}^{\mathcal H}$.

Now, let $r:\bigcup\mathcal H'\to W$ be a partial function, admissible with respect to $\mathcal H'$, local with respect to $\mathcal H$, and representing $\rho$. Composition with $r$ on the right is once again a total function from $\Partial(\mathcal H,\mathcal G,T)$ to $\Partial(\mathcal H',\mathcal G,T)$, because if $f\in\Partial(\mathcal H,\mathcal G,T)$, then it is admissible, and we have $\dd(f)\in\mathcal H$. Then since $r$ is local, we have $\dd(f\circ r)\in\mathcal H$, so that $f\circ r\in\Partial(\mathcal H',\mathcal G,T)$ -- i.e., it is an admissible partial function and we conclude that composition with the germ, $\rho$, is total and admissible. For locality, suppose now that we have $G\in\mathcal G$; we want to show that there is a basic set in the filter $\mathcal G^{\mathcal H}$ that will map into $\Partial(\mathcal H',\mathcal G,G)$, and we will show that $\Partial(\mathcal H,\mathcal G,G)$ will serve. Given $f\in\Partial(\mathcal H,\mathcal G,G)$, we have $f^{-1}(G)\in\mathcal H$. If we form $f\circ r$, then since $r$ is local, we see that 
$(f\circ r)^{-1}(G)=r^{-1}(\dd(r)\cap f^{-1}(G)))\in\mathcal H'$, so $f\circ r\in\Partial(\mathcal H',\mathcal G,G)$.
Since we have now shown that composition with $r$ on the right is admissible and local, so is its germ; thus, the arrow ${\mathcal G}^\rho:{\mathcal G}^{\mathcal H}\to{\mathcal G}^{\mathcal H'}$.

Clearly we have produced a functor
$(-_1)^{(-_2)}:\CatFil\times\CatFil^{\text{op}}\to\CatFil$.

\end{proof}

For this next definition, we make use of Notation~\ref{N:Fh}.

\begin{definition}
	Let $\mathcal H$ be a filter. For every pair of filters $\pair{\mathcal F}{\mathcal G}$, let \[\chi^{\mathcal H}_{\mathcal F,\mathcal G}:\CatFil(\mathcal F\opsquare\mathcal H,\mathcal G)\to\CatFil(\mathcal F,{\mathcal G}^{\mathcal H})\] be the total function mapping a germ $\kappa$ to 
	\begin{align*}\chi^{\mathcal H}_{\mathcal F,\mathcal G}(\kappa)&=[s\in F_{\mathcal F,\mathcal G,\dd(q)}\mapsto[w\in h_{\mathcal F,\mathcal G,\dd(q)}(s)\mapsto q(s,w)]/\mathcal H]/\mathcal F\in\CatFil(\mathcal F,{\mathcal G}^{\mathcal H})
	\end{align*}
	where $q\in\LPartial(\mathcal F\opsquare\mathcal H,\mathcal G)$ is a representative of the arrow (admissible, local germ) $\kappa$.
\end{definition}

Some propositions about this mapping, and Notation~\ref{N:Fh}, that we will need later:

\begin{lemma}\label{T:FhLemma} Let $\mathcal F$,  $\bar{\mathcal F}$, $\mathcal G$, $\tilde{\mathcal G}$, $\mathcal H$ be filters of subsets of sets $S$, $\bar S$, $T$. $\tilde T$, and $W$ respectively, and let $q\in\LPartial(\mathcal F\opsquare\mathcal H,\mathcal G)$, $\bar q\in\LPartial(\mathcal F,\bar{\mathcal F})$, $\tilde q\in\LPartial(\mathcal G,\tilde{\mathcal G})$, and $\hat q=q\circ(\bar q\opsquare_p\mathcal H)$.
	We have
	\begin{enumerate}
		\item\label{T:Fh1} $F_{\bar{\mathcal F},\mathcal H,\dd(\hat q)}=\bar q^{-1}(F_{\mathcal F,\mathcal H,\dd(q)})$;
		\item\label{T:Fh2} if $s\in F_{\bar{\mathcal F},\mathcal H,\dd(\hat q)}$, then $h_{\bar{\mathcal F},\mathcal H,\dd(\hat q)}(s)=h_{\mathcal F,\mathcal H,\dd(q)}(\bar q(s))$;
		\item\label{T:Fh3} $F_{\mathcal F,\mathcal H,\dd(\tilde q\circ q)}\subseteq F_{\mathcal F,\mathcal H,\dd(q)}$; and
		\item\label{T:Fh4} if $s\in F_{\mathcal F,\mathcal H,\dd(\tilde q\circ q)}$, then $h_{\mathcal F,\mathcal H,\dd(\tilde q\circ q)}(s)\subseteq h_{\mathcal F,\mathcal H,\dd(q)}(s)$.
	\end{enumerate}
\end{lemma}

\begin{proof}
	(1):
	\begin{align*} 	
		F_{\bar{\mathcal F},\mathcal H,\dd(\hat q)}
		&=\{\,s\in\bar S\mid\{\,w\in W\mid\pair sw\in \dd(\hat q)\,\}\in\mathcal H\,\}\\
		&=\{\,s\in\bar S\mid\{\,w\in W\mid\pair sw\in(\dd(\bar q)\times W)\cap (\bar q\mathrel{\opsquare_p}\mathcal H)^{-1}(\dd(q)) \,\}\in\mathcal H\,\}\\
		&=\{\,s\in\dd(\bar q)\mid\{\,w\in W\mid\pair sw\in(\bar q\mathrel{\opsquare_p}\mathcal H)^{-1}(\dd(q)) \,\}\in\mathcal H\,\}\\
		&=\{\,s\in\dd(\bar q)\mid\{\,w\in W\mid\pair{\bar q(s)}w\in \dd(q)\,\}\in\mathcal H\,\}\\
		&=\bar q^{-1}\left(\{\,s\in S\mid\{\,w\in W\mid\pair sw\in \dd(q)\,\}\in\mathcal H\,\}\right)\\
		&=\bar q^{-1}(F_{\mathcal F,\mathcal H,\dd(q)}).
	\end{align*}
	
	(2):
	\begin{align*}
		h_{\bar{\mathcal F},\mathcal H,\dd(\hat q)}(s)
		&=\{\,w\in W\mid\pair sw\in \dd(\hat q)\,\}\\
		&=\{\,w\in W\mid\pair{\bar q(s)}w\in \dd(q)\,\}\\
		&=h_{\mathcal F,\mathcal H,\dd(q)}(\bar q(s)).
	\end{align*}
	
	(3):
	\begin{align*}
		F_{\mathcal F,\mathcal H,\dd(\tilde q\circ q)}&=\{\,s\in S\mid\{\,w\in W\mid\pair sw\in\dd(\tilde q\circ q)\,\}\in\mathcal H\,\}\\
		&=\{\,s\in S\mid\{\,w\in W\mid\pair sw\in\dd(q)\cap q^{-1}(\dd\tilde q)\,\}\in\mathcal H\,\}\\
		&\subseteq\{\,s\in S\mid\{\,w\in W\mid\pair sw\in\dd(q)\,\}\in\mathcal H\,\}\\
		&=F_{\mathcal F,\mathcal H,\dd(q)}.
	\end{align*}
	
	(4):
	\begin{align*}
		h_{\mathcal F,\mathcal H,\dd(\tilde q\circ q)}(s)&=\{\,w\in W\mid\pair sw\in\dd(\tilde q\circ q)\,\}\\
		&=\{\,w\in W\mid\pair sw\in\dd(q)\cap q^{-1}(\dd\tilde q)\,\}\\
		&\subseteq\{\,w\in W\mid\pair sw\in\dd(q)\,\}\\
		&=h_{\mathcal F,\mathcal H,\dd(q)}(s).
	\end{align*}
\end{proof}

\begin{remark}
	Note that just as in our discussion of the natural transformation $\alpha$, there is not much mystery about where the partial functions we define send the elements in their domains.
\end{remark}

\begin{theorem} We have
	\begin{enumerate}
		\item $\chi^{\mathcal H}_{\mathcal F,\mathcal G}$ is a well-defined, one-one, and onto function;
		\item $\chi^{\mathcal H}$ is a
		natural isomorphism from the functor $\CatFil(-\opsquare\mathcal H,-):\CatFil^{\rm{op}}\times\CatFil\to\Set$ to the functor $\CatFil(-,-^{\mathcal H}):\CatFil^{\rm{op}}\times\CatFil\to\Set$, resulting in an adjunction
		\[\langle-\opsquare\mathcal
		H,-^{\mathcal H},\chi^{\mathcal
			H}\rangle:\CatFil\rightharpoonup\CatFil;\]
\item we have  a nonsymmetric closed structure
\[\langle\CatFil,\square,u,\alpha,\lambda,\varrho,\{\,\chi^{\mathcal H}\,\}_{\mathcal H\in\CatFil}\rangle\]
on the category $\CatFil$, where $u$, $\alpha=\alpha^g$, $\lambda=\lambda^g$, and $\varrho=\varrho^g$ are defined as in Section~\ref{S:FilMonoidal}.
	\end{enumerate}
\end{theorem}

\begin{proof}
	(1):  If $q:F\opsquare h\to T$ and $q':F'\opsquare h'$ are representatives of $\kappa$, then $q$ and $q'$ agree on $\hat F\opsquare \hat h$, for some $\hat F\in\mathcal F$ and $\hat h:\hat F\to\mathcal H$ such that $\hat F\opsquare\hat h\subseteq F\opsquare h\cap F'\opsquare h'$. Then for every $s\in\hat F$, $h(s)\cap h'(s)\subseteq \hat h(s)$, and we have
	\begin{align*}[w\in h(s)\mapsto q(s,w)]/\mathcal H&=[w\in \hat h(s)\mapsto q(s,w)]/\mathcal H\\
	&=[w\in h'(s)\mapsto q'(s,w)]/\mathcal H,\end{align*}
	proving that $\chi^{\mathcal H}_{\mathcal F,\mathcal G}(\kappa)$ does not depend on the choice of $q$.
	
	On the other hand, if we have the same $q$ and $q'$, except that $q\not\equiv q'$, then for any $\hat F\opsquare\hat h\in\mathcal F\opsquare\mathcal H$ with $\hat F\subseteq F\cap F'$ and $\hat h:\hat F\to\mathcal H$ with $\hat h(s)\subseteq h(s)\cap h'(s)$ for $s\in\hat F$, there is an $\hat s\in \hat F$ such that $[w\in\hat h(s)\mapsto q(s,w)]\not\equiv[w\in \hat h(s)\mapsto q(s,w)]$, which means that there is some $\pair{\hat s}{\hat w}\in\hat F\opsquare\hat h$ such that $q(\hat s,\hat w)\neq q'(\hat s,\hat w)$, and it follows that $[w\in\hat h(\hat s)\mapsto q(\hat s,w)]\not\equiv[w\in\hat h(\hat s)\mapsto q'(\hat s,w)]$ and $\chi^{\mathcal H}_{\mathcal F,\mathcal G}(q/\mathcal F)\neq\chi^{\mathcal H}_{\mathcal F,\mathcal G}(q'/\mathcal F)$. Thus, $\chi^{\mathcal H}_{\mathcal F,\mathcal G}$ is one-one.
	
	$\chi^{\mathcal H}_{\mathcal F,\mathcal G}(\kappa)$ is admissible  because it is the germ of a partial function with domain $F$. It is local, because if $X\in{\mathcal G}^{\mathcal H}$, there is a $G\in\mathcal G$ such that $\Gamma(\Partial(\mathcal H,\mathcal G,G))\subseteq X$, and we will have
	\[\chi^{\mathcal H}_{\mathcal F,\mathcal G}(q|_{F\times G})\in X;\]
	thus, $\chi^{\mathcal X}_{\mathcal F,\mathcal G}:\CatFil(\mathcal F\opsquare\mathcal H,\mathcal G\to\CatFil(\mathcal F,{\mathcal G}^{\mathcal H})$.
	
	To show $\chi^{\mathcal H}_{\mathcal F,\mathcal G}$ is onto, let us be given $\rho\in\CatFil(\mathcal F,{\mathcal G}^{\mathcal H})$, and an admissible, local partial function $r:\mathcal F\to{\mathcal G}^{\mathcal H}$ representing $\rho$, and define
	\[\bar\chi^{\mathcal H}_{\mathcal F,\mathcal G}(\rho)=[\pair sw\in\dd(r)\opsquare[s\mapsto\dd(y(s))]\mapsto y(s)(w)]/(\mathcal F\opsquare\mathcal H)\]
	where for every $s\in\dd(r)$, $y(s)$ is some representative of the germ of admissible partial functions $r(s)\in{\mathcal G}^{\mathcal H}$. Since we made choices here, this is a one-to-many relation.  Ignoring for the moment that this is not a function, and just fixing our choices in defining $\bar\chi^{\mathcal H}_{\mathcal F,\mathcal G}(\rho)$, we see that
	\begin{align*}
		\chi^{\mathcal H}_{\mathcal F,\mathcal G}\left(\bar\chi^{\mathcal H}_{\mathcal F,\mathcal F}(\rho)\right)
	&=\chi^{\mathcal H}_{\mathcal F,\mathcal G}\left(\left[\pair sw\in\dd(r)\opsquare\left[s\mapsto\dd(y(s))\right]\mapsto y(s)(w)\right]/(\mathcal F\opsquare\mathcal H)\right)\\
	&=\left[s\in\dd(r)\mapsto\left[w\in \dd(y(s)))\mapsto y(s)(w)\right]/\mathcal H\right]/\mathcal F\\
	&=\left[s\in\dd(r)\mapsto y(s)/\mathcal H\right]/\mathcal F\\
	&=\left[s\in\dd(r)\mapsto r(s)\right]/\mathcal F\\
	&=r/\mathcal F=\rho,\\
	\end{align*}
showing that $\chi^{\mathcal H}_{\mathcal F,\mathcal G}$ is onto.
	
(2): In order to show that $\chi^{\mathcal H}_{\mathcal F,\mathcal G}$ is natural in $\mathcal F$, it suffices to show that if in addition to having $\kappa$ as we have assumed, we have $\bar\kappa\in\CatFil(\bar{\mathcal F},\mathcal F)$, where $\bar{\mathcal F}$ is a filter of subsets of a set $\bar S$, then
\begin{equation}\mylabel{E:chi1}
\chi^{\mathcal H}_{\bar{\mathcal F},\mathcal G}(\kappa\circ (\bar
\kappa\opsquare\mathcal H))
=
(\chi^{\mathcal H}_{\mathcal F,\mathcal G}(\kappa))\circ \bar \kappa :\bar{\mathcal F}\to\mathcal G^{\mathcal H};
\end{equation}
indeed, if we set $\hat q=q\circ(\bar q\opsquare_p\mathcal H):\bar{\mathcal F}\opsquare\mathcal H\to\mathcal G$, we have
\begin{align*}
\chi^{\mathcal H}_{\bar{\mathcal F},\mathcal G}(\kappa\circ (\bar
\kappa\opsquare\mathcal H))
&=\chi^{\mathcal H}_{\bar{\mathcal F},\mathcal G}\left(q/(\mathcal F\opsquare\mathcal H)\circ(\bar q\opsquare_p\mathcal H)/(\bar{\mathcal F}\opsquare\mathcal H)\right)\\
&=\chi^{\mathcal H}_{\bar{\mathcal F},\mathcal G}\left(\hat q/\bar{\mathcal F}\opsquare\mathcal H\right)\\
&=[s\in F_{\bar{\mathcal F},\mathcal H,\dd(\hat q)}\mapsto[w\in h_{\bar{\mathcal F},\mathcal H,\dd(\hat q)}(s)\mapsto \hat q(s,w)]/\mathcal H]/\bar{\mathcal F}\\
&=[s\in F_{\bar{\mathcal F},\mathcal H,\dd(\hat q)}\mapsto[w\in h_{\bar{\mathcal F},\mathcal H,\dd(\hat q)}(s)\mapsto q(\bar q(s),w)]/\mathcal H]/\bar{\mathcal F}\\
&=[s\in F_{\bar{\mathcal F},\mathcal H,\dd(\hat q)}\mapsto[w\in h_{\mathcal F,\mathcal H,\dd(q)}(\bar q(s))\mapsto q(\bar q(s),w)]/\mathcal H]/\bar{\mathcal F}\\
&=[s\in\bar q^{-1}(F_{\mathcal F,\mathcal H,\dd(q)})\mapsto[w\in h_{\mathcal F,\mathcal H,\dd(q)}(\bar q(s))\mapsto q(\bar q(s),w)]/\mathcal H]/\bar{\mathcal F}]\\
&=\left([s\in F_{\mathcal F,\mathcal H,\dd(q)}\mapsto[w\in h_{\mathcal F,\mathcal H,\dd(q)}(s)\mapsto q(s,w)]/\mathcal H]\circ \bar q\right)/\bar{\mathcal F}\\
&=\left([s\in F_{\mathcal F,\mathcal H,\dd(q)}\mapsto[w\in h_{\mathcal F,\mathcal H,\dd(q)}(s)\mapsto q(s,w)]/\mathcal H]/\mathcal F\right)\circ (\bar q/\bar{\mathcal F})\\
&=(\chi^{\mathcal H}_{\mathcal F,\mathcal G}(\kappa))\circ \bar \kappa
\end{align*}
where we use Lemma~\ref{T:FhLemma}(\ref{T:Fh1}) and~(\ref{T:Fh2}).

To show that $\chi^{\mathcal H}_{\mathcal F,\mathcal G}$ is natural in $\mathcal G$, we must show that
\begin{equation}\mylabel{E:chi2}
\chi^{\mathcal H}_{\mathcal
	F,\tilde{\mathcal G}}(\tilde\kappa\circ \kappa)
	={\tilde\kappa}^{\mathcal H}\circ(\chi^{\mathcal H}_{\mathcal
	F,\mathcal G}(\kappa)):\mathcal F\to{\tilde{\mathcal G}}^{\mathcal H},
\end{equation}
for any $\tilde q/\mathcal G=\tilde \kappa:\mathcal G\to\tilde{\mathcal G}$, where $\tilde\kappa^{\mathcal H}$ is the usual shorthand for $\tilde \kappa^{1_{\mathcal H}}$. Considering the two sides of Equation~\ref{E:chi2}, we have

\begin{align*}
	\chi^{\mathcal H}_{\mathcal
		F,\tilde{\mathcal G}}(\hat\kappa\circ \kappa)
	&=\chi^{\mathcal H}_{\mathcal F,\mathcal G}(\tilde q/\mathcal G\circ q/(\mathcal F\opsquare\mathcal H))\\
	&=\chi^{\mathcal H}_{\mathcal F,\mathcal G}((\tilde q\circ q)/(\mathcal F\opsquare\mathcal H))\\
	&=[s\in F_{\mathcal F,\mathcal H,\dd(\tilde q\circ q)}\mapsto[w\in h_{\mathcal F,\mathcal H,\dd(\tilde q\circ q)}(s)\mapsto \tilde q(q(s,w))]/\mathcal H]/\mathcal F\\
	&=[s\in F_{\mathcal F,\mathcal H,\dd(\tilde q\circ q)}\mapsto[w\in h_{\mathcal F,\mathcal H,\dd(q)}(s)\mapsto \tilde q(q(s,w))]/\mathcal H]/\mathcal F\\
	&=[s\in F_{\mathcal F,\mathcal H,\dd(q)}\mapsto[w\in h_{\mathcal F,\mathcal H,\dd(q)}(s)\mapsto\tilde q(q(s,w))]/\mathcal H]/\mathcal F\\
	&=\left(\tilde q/\mathcal G\right)^{\mathcal H}\circ[s\in F_{\mathcal F,\mathcal H,\dd(q)}\mapsto[w\in h_{\mathcal F,\mathcal H,\dd(q)}(s)\mapsto q(s,w)]/\mathcal H]/\mathcal F\\
	&=\left(\tilde q/\mathcal G\right)^{\mathcal H}\circ\left(\chi^{\mathcal H}_{\mathcal
		F,\mathcal G}(q/(\mathcal F\opsquare\mathcal H))\right)\\
	&=\tilde\kappa^{\mathcal H}\circ(\chi^{\mathcal H}_{\mathcal
	F,\mathcal G}(\kappa)),
\end{align*}
using Lemma~\ref{T:FhLemma}(\ref{T:Fh3}) and~(\ref{T:Fh4}), which force germs with respect to $\mathcal F$ and $\mathcal H$ to be the same.

\end{proof}

\begin{remark}
	Note that each component $\eta^{\mathcal
		H}_{\mathcal F}:\mathcal F\to(\mathcal F\opsquare{\mathcal
		H})^{\mathcal H}$ of the unit natural
	transformation $\eta^{\mathcal H}$ is the
	germ of the function sending
	$s\in S$ to the germ of the function sending $w\in
	W$ to
	$\pair sw$. Each component $\varepsilon^{\mathcal
		H}_{\mathcal G}:{\mathcal G}^{\mathcal H}\opsquare\mathcal
	H\to\mathcal G$ of the counit natural
	transformation
	$\varepsilon^{\mathcal H}$ is the germ of the
	function sending $\pair{q/(\mathcal G^{\mathcal H}\opsquare\mathcal H)}w$ to $q(w)$ for $w\in\dd(q)$.
\end{remark}

\section{Applications}\mylabel{S:Applications}

\subsection{Uniform spaces on filters}
In \cite{reltunifonfil}, we discuss in detail the theory of uniform spaces with an underlying \emph{filter}, instead of an underlyiing \emph{set}. It should not come as a shock that there is a close relationship between $\CatFil$ and uniform spaces, when we consider that a uniformity on a set is defined as a \emph{filter} of entourages.    $\CatFil$ has a factorization system (Section~\ref{S:Factorization}), and we show in \cite{reltgeneq} that a certain list of properties of a category with factorization system support a theory of generalized equivalences, of which uniformities on a set are an example. These are the properties of $\CatFil$ (and its factorization system) that we proved in Section~\ref{S:CatFilProps}.

Thus, we consider in \cite{reltunifonfil} the category of uniform spaces on filters, which we continue to denote by $\Unif$ because it is so natural to do so.
A uniformity on a filter is a small generalization from a uniformity on a set, but one which manifests interesting new phenomena, especially when completion is considered. If we have a uniformity on a filter, the points of the filter that are not in the core play an interesting role.  They are not closed points, and can never be limits of a cauchy filter, but, there can be cauchy filters consisting entirely of non-core points.  These cauchy filters give rise to new elements when we apply the functor of (hausdorff) completion, $C:\Unif\to\Unif$.

This becomes especially important when we try to define a function space and make the category of uniform spaces on filters into a closed category.  Indeed this is possible, if we admit the possibility of a \emph{nonsymmetric} closed category.  The elements of the function space are germs of admissible partial functions, without regard to being local or uniformly continuous, which properties hold only for germs in the core of the function space object\footnote{And, the core of the function space coincides with the hom-set in the category $\Unif$. The core functor plays the role familiar from the theory of closed categories, symmetric or not, of a functor that takes function space objects and gives the corresponding hom-sets.}.  However, these non-core germs can give rise to new arrows when we complete the hom-objects and, using the fact{\rm\cite[Section~1]{reltclosed}} that completion is a \emph{monoidal functor}, form the category $\check C(\Unif)$, as defined in {\rm \cite[Section~3]{reltclosed}}. The new arrows include an inverse to the unit natural arrow from any uniform space into its completion, so that it becomes possible to work with the assumption that all spaces are complete.

\subsection{\texorpdfstring{$\Ind[\CatFil]$, a base category for Topological Algebra}{Ind(CatFil)}}
In \cite{reltgeneq}, we introduce the theory of $\Ind[\CatFil]$, a cartesian-closed category suitable for the study of Topological Algebra. In particular, for varieties of algebras that are congruence-modular, $\Ind[\CatFil]$ has properties that allow us to generalize Day's Theorem\cite{day}, something that is not possible\cite{weber} in $\Unif$.  Any hausdorff, compactly-generated topological algebra can be made into an object of $\Ind[\CatFil]$, and a factorization system (Section~\ref{S:Factorization}) in that category allows us to associate with that object, a structure lattice analogous to the congruence lattice or lattice of compatible uniformities on that algebra, which will be modular if the algebra belongs to a congruence-modular variety.

\printbibliography

\end{document}

%% file: abstract.tex
We explore the structure of $\CatFil$, the category of filters and germs of admissible partial functions. In particular, we show that $\CatFil$ is a nonsymmetric closed category, as defined in \cite{reltclosed}.